\documentclass[10pt]{article}
\title{Isoperimetric inequalities for the magnetic Neumann and Steklov problems with Aharonov-Bohm magnetic potential}
\author{Bruno Colbois\footnote{Universit\'e de Neuch\^atel\,, Institute de Math\'emathiques, Rue Emile Argand 11\,, 2000 Neuch\^atel\,, Switzerland. Email:\, bruno.colbois@unine.ch}, Luigi Provenzano\footnote{Sapienza Universit\`a di Roma\,, Dipartimento di Scienze di Base e Applicate per l'Ingegneria\,, Via Scarpa 16\,, 00161 Roma\,, Italy. Email:\, luigi.provenzano@uniroma1.it} and Alessandro Savo\footnote{Sapienza Universit\`a di Roma\,, Dipartimento di Scienze di Base e Applicate per l'Ingegneria\,, Via Scarpa 16\,, 00161 Roma\,, Italy. Email:\, alessandro.savo@uniroma1.it}}

\date{\today}
\usepackage{enumerate}
\usepackage{makeidx}
\usepackage{amssymb,amsmath,amsthm,graphicx,float}
\usepackage{xcolor}
\newtheorem{defi}{Definition} 
\newtheorem{thm}[defi]{Theorem}
\newtheorem{theorem}[defi]{Theorem}

\newtheorem{rem}[defi]{Remark}
 \newtheorem{prop}[defi]{Proposition}
\newtheorem{lemme}[defi]{Lemma}
\newtheorem{lemma}[defi]{Lemma}
\newtheorem{cor}[defi]{Corollary}
 
\newcommand{\twosystem}[2]{\left\{\begin{aligned} &#1\\ &#2\end{aligned}\right.}

\newcommand{\twomatrix}[4]{\begin{pmatrix} #1&#2\\#3&#4\end{pmatrix}}

\newcommand{\matrice}{\begin{pmatrix}}
\newcommand{\ok}{\end{pmatrix}}

\newcommand{\scal}[2]{\langle{#1},{#2}\rangle}

\newcommand{\abs}[1]{\lvert{#1}\rvert}

\newcommand{\reals}{{\mathbb R}}

\newcommand{\real}[1]{{\mathbb R}^{#1}}

\newcommand{\bd}{\partial}

\newcommand{\derive}[2]{\dfrac{\bd #1}{\bd#2}}
\newcommand{\deriven}[3]{\dfrac{\bd^{#1}#2}{{\bd#3}^{#1}}}

\newcommand{\R}{\mathbb R}

\begin{document}

\maketitle

\noindent
{\bf Abstract.} We discuss isoperimetric inequalities for the magnetic Laplacian on  bounded domains of $\mathbb R^2$ endowed with an Aharonov-Bohm potential. When the flux of the potential around the pole is not an integer, the lowest eigenvalue for the Neumann and the Steklov problems is positive. We establish isoperimetric inequalitites for the lowest eigenvalue in the spirit of the classical inequalities of Szeg\"o-Weinberger, Brock and Weinstock, the model domain being a disk with the pole at its center. We consider more generally domains in the plane endowed with a rotationally invariant metric, which include the spherical and the hyperbolic case.

\vspace{11pt}


\vspace{6pt}
%


\section{Introduction} The question of the \emph{isoperimetric inequalities} for the eigenvalues of the Laplacian (in particular for the first nonzero eigenvalue) is a long standing problem. Let us give a short and partial summary in the case of bounded domains of the Euclidean space which will be the main topic of the present paper. It began with the celebrated Faber-Krahn inequality \cite{Kr}: for  Dirichlet boundary conditions, among the open bounded domains of given volume, the {\it first} eigenvalue is minimized by the ball. For spaces of constant curvature the result can be found in \cite{chavel}. For Neumann boundary conditions, among the bounded open domains of given volume with Lipschitz boundary, the {\it second} eigenvalue (i.e., the first nonzero) is maximized by the ball. This is the Szeg\"o-Weinberger inequality \cite{szego_1,weinberger}. This inequality has been extended to bounded domains in spaces of constant curvature by Ashbaugh and Benguria \cite{ash_beng} (see also \cite{bandle}). For  Robin boundary conditions with positive parameter, the ball also realizes the minimum \cite{Da}. For other operators, similar results exist. For the Steklov problem, the {\it second} eigenvalue (the first nonzero) is maximized by the ball among the bounded open domains of given volume with Lipschitz boundary. This is the inequality of Brock \cite{Br}. However, if we consider the domains of $\R^2$ with boundary of given length, the second eigenvalue is maximized by the disk only among the simply connected domains. This is the inequality of Weinstock \cite{weinstock}. There exist annuli with larger second eigenvalue. Again, we refer to \cite{bandle,chavel} for more discussion and generalizations. Note that, even if we will not go in this direction, the maximization or minimization of higher eigenvalues is intensively studied, see for example \cite{BuHe2019} for the second nonzero eigenvalue of the Neumann problem and \cite{GiLa2021} for the third eigenvalue of the Robin problem.

\smallskip
In this paper we will be mainly concerned with the Neumann problem for the Aharonov-Bohm magnetic Laplacian on domains of $\mathbb R^2$ (see \eqref{A0}-\eqref{AB_N}) and with the corresponding Steklov problem (see \eqref{A0}-\eqref{AB}). In the case of the Neumann problem \eqref{AB_N}, we will also consider the Aharonov-Bohm magnetic Laplacian on domains of surfaces of revolution (in particular, the standard sphere $\mathbb S^2$ and the standard hyperbolic space $\mathbb H^2$). Most of the time, the first eigenvalue of this kind of problem is strictly positive and its study difficult. For example, for the magnetic Laplacian with constant non zero magnetic field and Dirichlet boundary condition in $\mathbb R^2$, it is known that the first eigenvalue is minimized by the disk among open domain of given area: this was shown in \cite{Er1} and the proof is quite involved. To our knowledge, a similar result is not known in $\mathbb S^2$ or $\mathbb H^2$. 
However, for the magnetic Laplacian with constant magnetic field and magnetic Neumann boundary condition, it is no longer true that the disk maximizes the first eigenvalue: even for simply connected domains, the question is open, see \cite[Question 1, Remark 2.4 and Proposition 3.3]{FH2}. More information can be found also in \cite[\S 4 and \S 5]{FH_book}. Still for the case of constant magnetic field, we mention \cite{laugesen1,laugesen2} for bounds on Dirichlet and Neumann eigenvalues of certain families of domains, and \cite{loto} for an isoperimetric inequality for the Robin problem.

\smallskip 

In the case of $\mathbb R^2$, we show that, among all domains of given area, the disk with the singularity of the magnetic field at the center is the unique maximizer of the {\it first} eigenvalue, which is positive provided that the flux is not an integer. This is a reverse Faber-Krahn inequality, that we obtain in the spirit of Szeg\"o-Weinberger \cite{szego_1,weinberger}. For the Steklov problem we prove two isoperimetric inequalities for the {\it first} eigenvalue, which, again, is positive if the flux is not an integer. These correspond to the inequalities of Weinstock and of Brock \cite{Br,weinstock}.

\smallskip

For the Neumann problem, we obtain similar results for domain of $\mathbb S^2$ and $\mathbb H^2$. The result will be a consequence of a general isoperimetric inequality for the Schr\"odinger operator on a manifold of revolution with radial, non-negative, and radially decreasing potential, that we will prove here (see Theorem \ref{sch}).

\smallskip

\smallskip We finally remark that the Faber-Krahn inequality for the magnetic Dirichlet problem with Aharonov-Bohm potential is trivial: the first eigenvalue is minimized by that of the usual Laplacian on the disk, among all bounded domains of given area.
 
\section{Notation and statement of results}

\medskip

Let $\Omega$ be a smooth bounded domain of $\mathbb R^2$ with a distinguished point $x_0=(a,b)$ and consider the one-form
 \begin{equation}\label{A0}
 A_0=-\frac{x_2-b}{(x_1-a)^2+(x_2-b)^2}dx_1+\frac{x_1-a}{(x_1-a)^2+(x_2-b)^2}dx_2.
 \end{equation}
 The one-form $A_{x_0,\nu}=\nu A_0$ will be called {\it Aharonov-Bohm potential with pole $x_0$ and flux $\nu$}. 
 Note that $A$ is smooth, closed, co-closed (hence harmonic) on $\real 2\setminus \{x_0\}$, and is singular at $x_0$; it gives rise to a zero magnetic field ($B=d{A_{x_0,\nu}}=0$).  Let $\Delta_{A_{x_0,\nu}}$ be the magnetic Laplacian with potential $A_{x_0,\nu}$: it is the operator 
$$
\Delta_{A_{x_0,\nu}} u=\Delta u+|{A_{x_0,\nu}}|^2u+2i\langle\nabla u, {A_{x_0,\nu}} \rangle
$$
acting on complex valued functions $u$ (the sign convention is that $\Delta u=-\sum_j \partial^2_{x_jx_j}u$). Of course, we can always assume that $x_0$ is the origin. 

\medskip

We will also consider the case when the ambient space is a two-dimensional manifold of revolution $(M,g)$ with pole $x_0$ and polar coordinates $(r,t)$, where $r$ is the distance to $x_0$. In this case, we consider the form $A_0=dt$, which is harmonic (closed and co-closed), and with flux $1$ around $x_0$.

\medskip

In this paper, we consider the eigenvalue problem for $\Delta_{A_{x_0,\nu}}$ with magnetic Neumann conditions:
\begin{equation}\label{AB_N}
\begin{cases}
\Delta_{A_{x_0,\nu}} u=\lambda u\,, & {\rm in\ } \Omega,\\
\langle\nabla u -iu {A_{x_0,\nu}},N\rangle=0\,, & {\rm on\ }\partial\Omega,
\end{cases}
\end{equation}

and also the magnetic Steklov eigenvalue problem:
\begin{equation}\label{AB}
\begin{cases}
\Delta_{A_{x_0,\nu}} u=0\,, & {\rm in\ } \Omega,\\
\langle\nabla u -iu {A_{x_0,\nu}},N\rangle=\sigma u\,, & {\rm on\ }\partial\Omega.
\end{cases}
\end{equation}

Here $N$ is the outer unit normal to $\partial\Omega$. With abuse of notation we still denote by ${A_{x_0,\nu}}$ the potential dual to the $1$-form $A_{x_0,\nu}$. We also denote by $\nabla^{A_{x_0,\nu}}u$ the vector field 
$$
\nabla^{A_{x_0,\nu}}u=\nabla u-iu{A_{x_0,\nu}}
$$
which is called {\it magnetic gradient}. Therefore, the magnetic  Neumann condition reads $\scal{\nabla^{A_{x_0,\nu}}u}{N}=0$, while the magnetic Steklov condition is $\scal{\nabla^{A_{x_0,\nu}}u}{N}=\sigma u$. 

\medskip

We will prove in Appendix \ref{functional} that each of these two problems admits an infinite discrete sequence of eigenvalues of finite multiplicity. We will denote by $\lambda_1(\Omega,A_{x_0,\nu})$ the first eigenvalue of Problem \eqref{AB_N} and by $\sigma_1(\Omega,A_{x_0,\nu})$ the first eigenvalue of Problem \eqref{AB}. These two eigenvalues are  non-negative for all $\nu\in\mathbb R$  and are strictly positive if and only if $\nu \not \in \mathbb Z$ (in particular, when $x_0\in\Omega^c$ and the flux of $A_0$ is $0$ in $\Omega$, we set $\nu=0$); in particular, when $\nu\in\mathbb Z$, the two spectra of problems \eqref{AB_N} and \eqref{AB} reduce to the corresponding spectra of the Laplacian $\Delta$  (i.e., when $A_{x_0,\nu}=0$, see Appendix \ref{functional}).

\medskip

In the sequel, we will often suppose that $\nu \not \in \mathbb Z$, so $\lambda_1(\Omega,A_{x_0,\nu})$ and $\sigma_1(\Omega,A_{x_0,\nu})$ are both positive.

The first result is a reverse Faber-Krahn inequality for the first eigenvalue of the Neumann problem whose proof is based on the well-known Szeg\"o-Weinberger approach \cite{szego_1,weinberger}. 

\smallskip

Through all the paper, by $|\Omega|$ we denote the Lebesgue measure of a smooth bounded domain $\Omega$, and by $|\partial\Omega|$ the length of its boundary.

\begin{thm} \label{ThmNeumann}Let $\Omega$ be a smooth bounded domain in $\mathbb R^2$ or $\mathbb H^2$, and let $A_{x_0,\nu}$ be the Aharonov-Bohm potential with pole at $x_0$ and flux $\nu$. Let $B(x_0,R)$ be the disk with center $x_0$ and radius $R$ such that $\vert B(x_0,R)\vert =\vert \Omega \vert$. Then
\begin{equation} \label{Neumann}
\lambda_1(\Omega,A_{x_0,\nu}) \le \lambda_1(B(x_0,R),A_{x_0,\nu});
\end{equation}
if $\nu\notin\mathbb Z$, equality holds if and only if $\Omega=B(x_0,R)$.
\end{thm}     

We will observe that Theorem \ref{ThmNeumann} extends to domains in the manifold $(\mathbb R^2,g)$ where $g$ is a complete, non-positively curved, rotationally invariant metric around $x_0$, the pole of the magnetic potential (see Section \ref{app_AB_Rev}).

This theorem will be a consequence of a more general result about an isoperimetric inequality for Schr\"{o}dinger operators on revolution manifolds with pole $x_0$ and radial potential $V$ that we will present in Section \ref{Schroedinger}. In fact, Theorem \ref{ThmNeumann} holds also in this setting, under suitable hypothesis on the function describing the density of the Riemannian metric in standard polar coordinates. 

\smallskip The case of the sphere $\mathbb S^2$ is more involved. We are able to show a similar result to Theorem \ref{ThmNeumann} only if the domain is contained in a hemisphere centered at the pole $x_0$.

\begin{thm} \label{ThmNeumannS}Let $\Omega$ be a smooth domain contained in a hemisphere centered at $x_0$, and let $A_{x_0,\nu}$ be the Aharonov-Bohm potential with pole at $x_0$ and flux $\nu$. Let $B(x_0,R)$ be the disk in $\mathbb S^2$ with center $x_0$ and radius $R$ such that $\vert B(x_0,R)\vert =\vert \Omega \vert$. Then
\begin{equation} \label{NeumannS}
\lambda_1(\Omega,A_{x_0,\nu}) \le \lambda_1(B(x_0,R),A_{x_0,\nu});
\end{equation}
if $\nu\notin\mathbb Z$, equality holds if and only if $\Omega=B(x_0,R)$.
\end{thm}     

Note that the analogous result for the second eigenvalue of the Neumann Laplacian is proved in \cite{ash_beng}. However, for simply connected domains we can do better.

\begin{thm} \label{ThmNeumannSS}Let $\Omega$ be a smooth simply connected domain in $\mathbb S^2$ with $|\Omega|\leq 2\pi$ and $-x_0\notin\Omega$, and let $A_{x_0,\nu}$ be the Aharonov-Bohm potential with pole at $x_0$ and flux $\nu$. Let $B(x_0,R)$ be the disk in $\mathbb S^2$ with center $x_0$ and radius $R$ such that $\vert B(x_0,R)\vert =\vert \Omega \vert$. Then
\begin{equation} \label{NeumannSS}
\lambda_1(\Omega,A_{x_0,\nu}) \le \lambda_1(B(x_0,R),A_{x_0,\nu});
\end{equation}
if $\nu\notin\mathbb Z$, equality holds if and only if $\Omega=B(x_0,R)$.
\end{thm}

The next result is the analogous of Brock's inequality \cite{Br} for the first  Steklov eigenvalue on planar domains:

\begin{thm} \label{ThmBrock}Let $\Omega$ be a smooth bounded domain in $\mathbb R^2$ and let $A_{x_0,\nu}$ be the Aharonov-Bohm potential with pole at $x_0$ and flux $\nu$. Let $B(x_0,R)$ be the disk with center $x_0$ and radius $R$ such that $\vert B(x_0,R)\vert =\vert \Omega \vert$. Then
\begin{equation} \label{Brock}
\sigma_1(\Omega, A_{x_0,\nu})\leq \sigma_1(B(x_0,R),A_{x_0,\nu})
=\dfrac{\sqrt{\pi}}{\sqrt{\abs{\Omega}}}\inf_{k\in\mathbb Z}\abs{\nu-k}.
\end{equation}
If $\nu\notin\mathbb Z$,  equality holds if and only if $\Omega=B(x_0,R)$.
\end{thm}

\medskip

Finally, we prove the analogue of Weinstock's inequality \cite{weinstock}:

\begin{thm} \label{ThmWeinstock}
Let $\Omega$ be a smooth bounded and simply connected domain in $\mathbb R^2$ and let $A_{x_0,\nu}$ be the Aharonov-Bohm potential with pole at $x_0$ and flux $\nu$.  Let $B(x_0,R)$ be the disk with center $x_0$ and radius $R$ such that $\vert \partial B(x_0,R)\vert =\vert \partial\Omega \vert$. Then
\begin{equation} \label{Weinstock}
\sigma_1(\Omega,A_{x_0,\nu})\leq \sigma_1(B(x_0,R),A_{x_0,\nu})= \dfrac{2\pi}{\abs{\bd\Omega}}\inf_{k\in\mathbb Z}\abs{\nu-k}.
\end{equation}
If $\nu\notin\mathbb Z$,  equality holds if and only if $\Omega=B(x_0,R)$.
\end{thm}

Note that the upper bounds of Theorems \ref{ThmBrock} and \ref{ThmWeinstock} correctly reduce to zero whenever the flux is an integer.   

\smallskip

We stated Theorem \ref{ThmWeinstock} for planar, simply connected domains, however it extends to any Riemannian surface with boundary. 

\medskip

We conclude this section with a few remarks. It is natural to ask what happens for the second eigenvalue of \eqref{AB_N} and \eqref{AB}, at least on planar domains. One immediately observes that Theorems \ref{ThmNeumann} and \ref{ThmBrock} no longer hold, in the sense that the ball punctured at the origin is not a maximiser. In fact, $\lambda_2(B(x_0,R),A_{x_0,\nu})=\frac{(z_{1-\inf_{k\in\mathbb Z}\abs{\nu-k},1}')^2}{R^2}<\frac{(z_{1,1}')^2}{R^2}$ when $\nu\notin\mathbb Z$. Here $z_{\mu,1}'$ denotes the first positive zero of the derivative of the Bessel function $J_{\mu}$ (see Appendix \ref{disk}). We recall that $\frac{(z_{1,1}')^2}{R^2}$ is exactly the second Neumann eigenvalue of the Laplacian on a ball of radius $R$. Analogously, we have $\sigma_2(B(x_0,R),A_{x_0,\nu})=\dfrac{1-\inf_{k\in\mathbb Z}\abs{\nu-k}}{R}<\frac{1}{R}$ (see Appendix \ref{disk_S}), and $\frac{1}{R}$ is the second Steklov eigenvalue of the Laplacian on $B(x_0,R)$. However, for  problem \eqref{AB_N} (\eqref{AB}), it can be shown that the disjoint union of two balls with suitable radii, and one of them centered at the pole, and total area $\pi$, has second eigenvalue strictly greater than that of the standard Neumann (Steklov) eigenvalue on $B(0,1)$. Therefore we are left with the following

\medskip

{\bf Open problem 1.} Find (if exists) a maximiser for the second Neumann (Steklov) Aharonov-Bohm eigenvalue among all smooth bounded domains in $\mathbb R^2$. 

\medskip

As for inequality \eqref{Weinstock}, preliminary calculations show that it holds for all circular annuli in $\mathbb R^2$, but it fails in the case of long  cylinders. In fact, when $\Omega=\mathbb S^1\times (-L,L)$, the first Steklov eigenvalue is given by $\inf_{k\in\mathbb Z}\abs{\nu-k}\tanh\left(\inf_{k\in\mathbb Z}\abs{\nu-k}L\right)$, hence \eqref{Weinstock} does not hold for $L>L_0$, with $L_0$ sufficiently large. We are left with the following

\medskip

{\bf Open problem 2.} Does inequality \eqref{Weinstock} hold for all doubly connected domains of the plane?

\medskip

The present paper is organized as follows. In Section \ref{Schroedinger} we prove an isoperimetric inequality for the first (positive) Neumann eigenvalue of the Schr\"odinger operator $\Delta+V$ on domains in manifolds of revolution, under suitable hypothesis on the potential $V$ and on the density of the Riemannian metric  (Theorem \ref{sch}). In Section \ref{app_AB}, Theorem \ref{sch} is applied to the magnetic Neumann spectrum. In particular, in Subsection \ref{app_AB_Rev} the reverse Faber-Krahn inequality is proved for manifolds of revolution (Theorem \ref{ab}). As a consequence, we prove that it holds  for domains in $\mathbb R^2$ and $\mathbb H^2$, (Corollary \ref{plane_AB_N_1}). In Subsection \ref{app_AB_S} it is proved for spherical domains contained in a hemisphere centered at the pole (Theorem \ref{ab_S}). In Subsection \ref{sec:szego} we prove the isoperimetric inequality for spherical simply connected domains with area less than $2\pi$ (Theorem \ref{ab_S_szego}). In Section \ref{sec:brockweinstock} we prove Brock's inequality for planar domains (Theorem \ref{thm:brock}), and Weinstock's inequality for planar domains (Theorem \ref{thm:weinstock}).

\medskip
We have included in this article a quite complete set of appendices, where we discuss  the functional and geometrical setting for the magnetic problems that we consider. In particular, we will compute explicitly the Neumann and Steklov spectrum for the unit disk.

\medskip

 In Appendix \ref{functional} we provide the basic spectral theory for problems \eqref{AB_N} and \eqref{AB}. Appendix \ref{sec:AB_revolution} contains a more explicit description of the eigenvalues and the eigenfunctions of the magnetic Neumann and Steklov problems on disks in manifolds of revolution (see Appendices \ref{sub:rev_N} and \ref{rev_S}). These facts, which have an interest on their own, are crucial for the proofs of the main Theorems. In Appendices \ref{disk} and \ref{disk_S} we describe the eigenfunctions and eigenvalues on disks in $\mathbb R^2$. Finally, in Appendix \ref{sec:conformal} we prove the conformal invariance of the Aharonov-Bohm energy which is crucial in the proof of Weinstock's inequality.


\section{Isoperimetric inequality for Schr\"odinger operators}\label{Schroedinger} 

In this section,  $\Omega$ will be a bounded smooth domain in a $n$-dimensional manifold of revolution $(M,g)$ with pole $x_0$. With $D$ we denote the diameter of $M$ (which can be infinite) and with $D_{\Omega}$ we denote the diameter of $\Omega$. 

\medskip

We recall that a smooth $n$-dimensional Riemannian manifold $(M,g)$ with a distinguished point $x_0$ is called a {\it revolution manifold with pole $x_0$} if $M\setminus\{x_0\}$ is isometric to $(0,D]\times\mathbb S^{n-1}$ whose metric is, in normal coordinates based at the pole, $g=dr^2+\Theta(r)^2 g_{\mathbb S^{n-1}}$, for $r\in (0,D)$. Here $\Theta(0)=\Theta''(0)=0$, $\Theta'(0)=1$, and $g_{\mathbb S^{n-1}}$ is the standard metric on the $n-1$-dimensional sphere. The density of the Riemannian metric on $M$ in normal coordinates is given by $\sqrt{{\rm det}\,g}=\Theta^{n-1}(r)=\theta(r)$.
 
 It is known that, for space forms of constant curvature $K=0,-1,1$ we have:
$$
\theta(r)=
\begin{cases}r^{n-1} & \text{if\  }K=0\\
\sinh^{n-1}(r) & \text{if\  }K=-1\\
\sin^{n-1}(r)& \text{if \ }K=1
\end{cases}
$$
In general, we have $\theta>0$ on $(0,D)$. We refer to Appendix \ref{sec:AB_revolution} for more information on manifolds of revolution.

\medskip

We discuss here an isoperimetric inequality for the first eigenvalue of the Schr\"odinger operator:
\begin{equation}\label{schr}
\begin{cases}
\Delta u +Vu=\lambda u\,, & {\rm in\ }\Omega\\
\langle\nabla u,N\rangle=0\,, & {\rm on\ }\partial\Omega.
\end{cases}
\end{equation}

Note that the results of this section can be applied to manifolds of revolution of any dimension $n\geq 2$.

\medskip

{\bf Assumptions on the potential $V$.}
\begin{enumerate}
\item The potential $V$ is smooth on $M\setminus\{x_0\}$, non-negative and radial with respect to $x_0$, that is, $V=V(r)$;
\item $V$ is non-increasing on $(0,D_{\Omega})$ : $V'(r)\leq 0$ on $(0,D_{\Omega})$;
\item $\theta' V'+2V^2\theta\leq 0$ on $(0,R)$, where  $R>0$ is such that $|B(x_0,R)|=|\Omega|$;
\item there exists a first eigenfunction $u$ of \eqref{schr} on $B(x_0,R)$ which is non-negative, radial and non-decreasing in the radial direction: $u'\geq 0$.

\end{enumerate}

We consider the following number: 
\begin{equation}\label{minmax_V}
\lambda_1(\Omega,\Delta+V)=\inf_{0\ne u\in H^1_V(\Omega)}\frac{\int_{\Omega}|\nabla u|^2+Vu^2}{\int_{\Omega}u^2},
\end{equation}
where $H^1_V(\Omega)=\{u\in H^1(\Omega):V^{1/2}u\in L^2(\Omega)\}$, and $H^1(\Omega)$ is the standard Sobolev space of square integrable functions with square integrable weak first derivatives. Since $V$ is non-negative, the infimum in \eqref{minmax_V}  exists and is non-negative.  We are ready to state the main result of this section.

\begin{theorem}\label{sch}Let $\Omega$ be a smooth bounded domain in a manifold of revolution $M$ with pole at $x_0$. Let $B=B(x_0,R)$ be the ball centered at $x_0$ with the same volume of $\Omega$. Let Assumptions 1-4 hold. Then
$$
\lambda_1(\Omega,\Delta+V)\leq \lambda_1(B(x_0,R),\Delta+V).
$$
Equality holds if and only if $\Omega=B(x_0,R)$.
\end{theorem}

If the spectrum of \eqref{schr} is discrete in its lower portion, the number $\lambda_1(\Omega,\Delta+V)$ is  the first eigenvalue.  This is the case of regular potentials (e.g., $V\in L^{n/2}$ for $n\geq 3$ or $V\in L^{1+\delta}$, $\delta>0$ for $n=2$), but also of singular potentials of the form $\frac{\nu^2}{r^2}$ (inverse-square potentials). In both these cases, the whole spectrum is purely discrete and made of non-negative eigenvalues of finite multiplicity diverging to $+\infty$.

\medskip

Let now $B(x_0,R)$ be the ball of radius $R$ centered at  the pole $x_0$ and assume that there exists a first eigenfunction of \eqref{schr} on $B(x_0,R)$ which is non-negative (and therefore radial, as $V$ is radial) and non-decreasing with respect to $r$. Let us denote this function by $u=u(r)$. It satisfies
\begin{equation}\label{radialtheta}
\begin{cases}
{u''+\dfrac{\theta'}{\theta}u'+(\lambda-V)u=0}\,, & {\rm in\ }(0,R)\\
{u'(R)=0},
\end{cases}
\end{equation}
where $\lambda=\lambda_1(B(x_0,R),\Delta+V)$ is the first eigenvalue.

\medskip

In order to prove Theorem \ref{sch} we need the following lemma:

\begin{lemme}\label{lem_rev}
Let $u=u(r)$ be a solution of \eqref{radialtheta} such that $u\geq 0$ and $u'\geq 0$ on $(0,R)$. 
Let
$$
F(r)=u'(r)^2+Vu(r)^2.
$$
If $V'\leq 0$ on $(0,R)$ and  
$
\theta' V'+ 2V^2\theta\leq 0
$ on $(0,R)$,
then one has:
$$
F'(r)\leq 0
$$
on $(0,R)$.
\end{lemme}

\begin{proof} One has:
\begin{multline*}
F'=2u'u''+V'u^2+2Vuu'
=2u'\Big(-\frac{\theta'}{\theta}u'-(\lambda -V)u\Big)+V'u^2+2Vuu'\\
=-2\frac{\theta'}{\theta}u'^2-2\lambda uu'+V'u^2+4Vuu'
\leq -2\frac{\theta'}{\theta}u'^2+V'u^2+4Vuu'
\end{multline*}
because $u\geq 0$ and $u'\geq 0$. Now:
$$
V'u^2+4Vuu'=V'\Big(u+2\frac{V}{V'}u'\Big)^2-4\frac{V^2}{V'}u'^2,
$$
and we have:
$$
F'\leq -2\Big(\dfrac{\theta'}{\theta}+2\dfrac{V^2}{V'}\Big)u'^2+V'\Big(u+2\frac{V}{V'}u'\Big)^2.
$$
As $V'\leq 0$ we conclude:
$$
F'\leq -2\Big(\dfrac{\theta'}{\theta}+2\dfrac{V^2}{V'}\Big)u'^2.
$$
If $\theta' V'+2V^2\theta\leq 0$ then, dividing by $\theta V'$ (which is non-positive) we indeed have
$$
\dfrac{\theta'}{\theta}+2\dfrac{V^2}{V'}\geq 0
$$
which guarantees that $F'\leq 0$.

\end{proof}

\begin{proof}[Proof of Theorem \ref{sch}]
Define the radial function $f:M\to\reals$ as follows:
$$
f(r)=\twosystem{u(r)\quad\text{for $r\leq R$}}
{u(R)\quad\text{for $r\geq R$}}
$$
We note that, by construction, $f_{|_{\Omega}}\in H^1_V(\Omega)$, therefore it is possible to use as test function in \eqref{minmax_V}.
\medskip

We start by observing that, by assumption $\abs{\Omega\cap B^c}=\abs{\Omega^c\cap B}$, so that, since $u$ is increasing, we have $u(r)\leq u(R)$ and
$$
\int_{\Omega}f^2\geq\int_{B}u^2
$$
In  fact:
\begin{multline}\label{chain1}
\int_{\Omega}f^2=\int_{\Omega\cap B}f^2+\int_{\Omega\cap B^c}f^2
=\int_{\Omega\cap B}u^2+u(R)^2\abs{\Omega\cap B^c}\\
=\int_{\Omega\cap B}u^2+u(R)^2\abs{\Omega^c\cap B}
\geq\int_{\Omega\cap B}u^2+\int_{\Omega^c\cap B}u^2
=\int_{\Omega}u^2
\end{multline}

We have to control the energy. Since $F(r)=u'^2(r)+V(r)u^2(r)$ is decreasing, $f$ is constant, equal to $u(R)$ on $\Omega\cap B^c$, $V'\leq 0$ on $(0,D_{\Omega})$, and $u(R)\leq u(r)$ on $B^c$:
\begin{multline}\label{chain2}
\int_{\Omega\cap B^c}\abs{\nabla f}^2+Vf^2=u(R)^2\int_{\Omega\cap B^c}V\leq u(R)^2V(R)|\Omega\cap B^c|\\
=u(R)^2V(R)|\Omega^c\cap B|\leq F(R)|\Omega^c\cap B|\leq \int_{\Omega^c\cap B}F,
\end{multline}
where we have used the monotonicity of $V$ in the first inequality and the monotonicity of $F$ in the last inequality. Therefore, from \eqref{chain1}, \eqref{chain2} and from the fact that $F=\abs{\nabla f}^2+Vf^2=\abs{\nabla u}^2+Vu^2$ on $B$, we deduce:
\begin{multline*}
\lambda_1(\Omega,\Delta+V)\int_{B}u^2\leq \lambda_1(\Omega,\Delta+V)\int_{\Omega}f^2
\leq\int_{\Omega}\abs{\nabla f}^2+Vf^2\\
=\int_{\Omega\cap B}\abs{\nabla f}^2+Vf^2+\int_{\Omega\cap B^c}\abs{\nabla f}^2+Vf^2
\leq \int_{\Omega\cap B}F+\int_{\Omega^c\cap B}F
=\int_BF\\
=\int_B \abs{\nabla u}^2+Vu^2
=\lambda_1(B,\Delta+V)\int_Bu^2
\end{multline*}
and the assertion follows.
\end{proof}

\begin{rem}
Note that Assumption $3$ may look quite involved. However, when we will apply Theorem \ref{sch} to the particular case of the Aharonov-Bohm operator we will choose $V=\frac{\nu^2}{\theta^2}$ and condition $3$ will take a simpler and more natural form (see Theorem \ref{ab}).
\end{rem}

\section{Application to the Aharonov-Bohm spectrum of the Neumann problem}\label{app_AB}

We apply now the results of Section \ref{Schroedinger} to the lowest eigenvalue of problem \eqref{AB_N}. We consider first the general case of manifolds of revolution, then we concentrate on $\mathbb R^2$, $\mathbb H^2$ and $\mathbb S^2$.

\subsection{Aharonov-Bohm spectrum on domains of manifolds of revolution}\label{app_AB_Rev}

We take a $2$-dimensional manifold of revolution $M^2$ with pole $x_0$, and  $\theta(r)$ the density of the Riemannian metric in polar coordinates $(r,t)$ around the pole.  The $1$-form  $A_{x_0,\nu}=\nu\,dt$ is closed, harmonic, and has flux $\nu$ around $x_0$. It will be called Aharonov-Bohm potential with flux $\nu$. 

\smallskip

On a smooth bounded domain $\Omega$ of $M^2$ we have that the spectrum of $\Delta_{A_{x_0,\nu}}$ with Neumann condition, namely problem \eqref{AB_N}, is made of an increasing sequence of non-negative eigenvalues of finite multiplicity diverging to $+\infty$ (see Appendix \ref{functional}).

\smallskip
For all radial functions $u=u(r)$ we have
$$
\abs{A_{x_0,\nu}}^2=\dfrac{\nu^2}{\theta^2}, \quad \scal{\nabla u}{A_{x_0,\nu}}=0,
$$
therefore $\Delta_{A_{x_0,\nu}}$ applied to a real, radial function $u=u(r)$, writes:
$$
\Delta_{A_{x_0,\nu}}u=\Delta u+Vu
$$
where $V=\dfrac{\nu^2}{\theta^2}$.

In Appendix \ref{sub:rev_N} we will prove that for a disk centered at the pole $B(x_0,R)$ the spectrum can be written in terms of that of the union a countable family of Sturm-Liouville problems indexed by an integer $k$ (see Lemma \ref{lemma_N}). In particular, the first eigenvalue, denoted by $\lambda_1(\Omega,A_{x_0,\nu})$, is non-negative, and is positive if and only if $\nu\notin\mathbb Z$.

\smallskip

As explained in Appendix \ref{sub:rev_N}, thanks to gauge invariance we can take $\nu\in (0,\frac 12]$ and in that case the first eigenfunction is real and radial.

\smallskip

We denote it by $u=u(r)$. Moreover, we prove in Appendix \ref{sub:rev_N} that $u>0$ and $u'>0$ for all $R\in(0,\bar R)$, where $\bar R$ is the first zero of $\theta'$ (see Theorem \ref{technical_sphere_0}).

\medskip

We can apply Theorem \ref{sch}, taking $V=\dfrac{\nu^2}{\theta^2}$. The conditions
$$
V'\leq 0, \quad \theta'V'+2V^2\theta\leq 0
$$ 
reduce to the conditions
$$
\theta'\geq 0, \quad \theta'^2\geq \nu^2.
$$

\medskip

{\bf Assumption.} Through all this subsection, we shall always assume $\nu\in(0,\frac{1}{2}]$.

\medskip 
We therefore have the following:

\begin{theorem}\label{ab} Let $\Omega$ be a smooth bounded domain with diameter $D_{\Omega}$ in a revolution manifold $M^2$ with pole $x_0$, and  with density $\theta$, and let $B(x_0,R)$ be the disk centered in $x_0$ with the same volume as $\Omega$.  Assume $\nu\in(0,\frac{1}{2}]$. If
\begin{enumerate}[i)]
\item $\theta'\geq 0$ on $(0,D_{\Omega})$,
\item $\theta'^2\geq \nu^2$ on $(0,R)$,
\end{enumerate}
then
\begin{equation}\label{ineqV}
\lambda_1(\Omega,A_{x_0,\nu})\leq \lambda_1(B(x_0,R),A_{x_0,\nu}).
\end{equation}
Equality holds if and only if $\Omega=B(x_0,R)$. 
\end{theorem}

When the flux is $0$, inequality \eqref{ineqV} reduces to the identity $0=0$ for all domains $\Omega$.

Note that Theorem \ref{ab} provides sufficient conditions to have a reverse Faber-Krahn inequality on a manifold of revolution, which are quite simple to understand. Theorem \ref{ab} works well, as we shall see, in the case of $\mathbb R^2$ and $\mathbb H^2$. However, in some cases (e.g., spherical domains), condition $ii)$ is somehow restrictive. In Appendix \ref{sub:rev_N} we prove Theorem \ref{technical_sphere}, where we show that $ii)$ can be replaced by some other (more involved) condition, which, in the case of the sphere, turns out to be less restrictive than $ii)$. Nevertheless, we decide to keep Theorem \ref{ab} here for two reasons: it is simpler and has immediate application in many contexts; it is a consequence of a more general result valid for Schr\"odinger operators, namely Theorem \ref{sch}, which we believe has an interest per se. 



\medskip

We now assume that $M^2$ has infinite diameter, so that we can identify it with the manifold $(\mathbb R^2,g)$, where $g$ is a complete, rotationally invariant metric around $x_0$, and consider the $1$-form $A_{x_0,\nu}$. If $\theta(r)$ is the density of the Riemannian measure, then it is well-known that the Gaussian curvature of $(\mathbb R^2,g)$ is given by $K=-\frac{\theta''}{\theta}$. Assuming $K\leq 0$, we will get $\theta''\geq 0$, and since $\theta'(0)=1$, we immediately obtain $\theta'(r)\geq 1$ for all $r>0$. The assumptions $i)$ and $ii)$ of Theorem \ref{ab} are met, thus we have the following:





\begin{cor}\label{plane_AB_N_1} Let $\Omega$ be a smooth bounded domain in the manifold of revolution $(\mathbb R^2,g)$ with pole at $x_0$ and non-positive Gaussian curvature. Let $B(x_0,R)$ be the disk of radius $R$ centered at $x_0$ such that $|B(x_0,R)|=|\Omega|$. Then
$$
\lambda_1(\Omega,A_{x_0,\nu})\leq \lambda_1(B(x_0,R),A_{x_0,\nu}),
$$
with equality if and only if $\Omega=B(x_0,R)$.
\end{cor}

In particular, Corollary \ref{plane_AB_N_1} applies to $\mathbb R^2$ with its Euclidean Riemannian metric, and to $\mathbb H^2$. Since all points in $\mathbb R^2$ and $\mathbb H^2$ can be chosen as poles of the manifold, this gives a proof of Theorem \ref{ThmNeumann}.

\smallskip

Another consequence of Theorem \ref{ab} is the following

\begin{cor}
Let $B(p,R)$ be a disk in $\mathbb R^2$ or $\mathbb H^2$, punctured at $x_0\in B(p,R)$. Then
$$
\lambda_1(B(p,R),A_{x_0,\nu})\leq\lambda_1(B(x_0,R),A_{x_0,\nu}),
$$
that is, among all disks of the same measure, the first eigenvalue is maximized by the disk punctured at its center.
\end{cor}

\subsection{Aharonov-Bohm spectrum on domains of $\mathbb S^2$}\label{app_AB_S}

\medskip

For the standard sphere $\mathbb S^2$ of curvature $1$ we have
$$
\theta(r)=\sin(r),
$$
hence a direct application of Theorem \ref{ab} is possible only for domains $\Omega$ such that $|\Omega|\leq|B(x_0,\pi/3)|$ and contained in a hemisphere centered at $x_0$.

\medskip

{\bf Assumption.} Through all this subsection, we shall always assume $\nu\in(0,\frac{1}{2}]$.

\medskip

\medskip 

The condition  $|\Omega|\leq|B(x_0,\pi/3)|$ ensures $\theta'^2\geq\nu^2$ on $B(x_0,R)$, in fact $R\leq\frac{\pi}{3}$ and then $\cos'(r)^2\geq\frac{1}{4}$.
\smallskip

The condition that $\Omega$ is contained in a hemisphere centered at $x_0$ ensures that $\theta'\geq 0$ on $\Omega\cup B(x_0,R)$. 

\medskip

Note that, restrictions to the class of spherical domains for which one usually proves isoperimetric inequalities are natural and common. For example, the Szeg\"o-Weinberger  inequality for the second eigenvalue of the Neumann Laplacian on spherical domains is proved under the assumption that the domain is contained in a hemisphere (\cite{ash_beng}) or that it is simply connected with total area less than $2\pi$ (\cite{bandle}). In the first case, the approach is that of Weinberger for planar domains \cite{weinberger} (which is the one that we have used up to now), while in the second case the approach is that of Szeg\"o by conformal transplantation \cite{szego_1}.

\medskip

The hypothesis of Theorem \ref{ab} are not sufficient to cover the case of spherical domains contained in a hemisphere centered at the pole of the magnetic field, which is the natural counterpart of the results in \cite{ash_beng}.  However, we prove in Theorem \ref{technical_sphere} that the function $F(r)=u'(r)^2+\frac{\nu^2}{\theta(r)^2}u(r)^2$ is decreasing in $(0,R)$ for all $R\in (0,\bar R)$ under suitable assumptions, even if $\theta'<\nu$. Recall that $\bar R$ is the first zero of $\theta'$ and $R$ is such that $|B(x_0,R)|=|\Omega|$. The fact that $F'\leq 0$, together with $\theta'> 0$, implies the hypotheses of Theorem \ref{sch} for $V=\frac{\nu^2}{\theta^2} $ and henceforth that of of Theorem \ref{schr}. In particular, the assumptions of Theorem \ref{technical_sphere} for $\mathbb S^2$ reduce to
\begin{enumerate}[i)]
\item $\theta'>0$
\item $
\nu(\nu+1)\theta^2-\nu^2+\nu^2(\theta')^2+\nu\theta\theta''\geq 0$,
\end{enumerate}
which are clearly satisfied as long as $R\leq\frac{\pi}{2}$.
In view of this, we only need to assume that $\Omega$ is contained in a hemisphere centered at the pole $x_0$. Then we have the expected result

\begin{theorem}\label{ab_S} Let $\Omega$ be a smooth bounded domain contained in a hemisphere centered at $x_0$. Let $B(x_0,R)$ be the disk centered in $x_0$ with the same volume as $\Omega$.  Then
$$
\lambda_1(\Omega,A_{x_0,\nu})\leq \lambda_1(B(x_0,R),A_{x_0,\nu}),
$$
with equality if and only if $\Omega=B(x_0,R)$. 
\end{theorem}

Note that this result coincides  with that of Ashbaugh-Benguria \cite{ash_beng}. 

\smallskip

We also deduce the following
\begin{cor}
Let $B(p,R)\subset\mathbb S^2$ be a spherical disk punctured at $x_0\in B(p,R)$. If $R\leq\frac{\pi}{4}$, then
$$
\lambda_1(B(p,R),A_{x_0,\nu})\leq\lambda_1(B(x_0,R),A_{x_0,\nu}),
$$
that is, among all disks of the same measure and radius smaller than $\frac{\pi}{4}$, the first eigenvalue is maximized by the disk punctured at its center.
\end{cor}

\subsection{Szeg\"o's isoperimetric inequality on spheres}\label{sec:szego}

Theorem \ref{ab_S} is valid for all bounded domains which are contained in a hemisphere centered at the pole $x_0$ of $A_{x_0,\nu}$. We have used the Weinberger's argument for its proof. If we take the Szeg\"o's point of view, we are able to extend the result to the class of simply connected domains on the sphere with area less than $2\pi$. Namely, we have:

\begin{theorem}\label{ab_S_szego} Let $\Omega$ be a bounded and simply connected domain in $\mathbb S^2$ with $|\Omega|\leq 2\pi$ and $-x_0\notin\Omega$, and let $B(x_0,R)$ be the disk in $\mathbb S^2$ centered at $x_0$ with $|B(x_0,R)|=|\Omega|$.  Then
$$
\lambda_1(\Omega,A_{x_0,\nu})\leq \lambda_1(B(x_0,R),A_{x_0,\nu}).
$$
If $\nu\notin\mathbb Z$, equality holds if and only if $\Omega=B(x_0,R)$. 
\end{theorem}
\begin{proof}
Through the stereographic projection $f$ we identify a point $re^{i t}\in\mathbb S^2$ ($r$ is the distance to $x_0$, and $t$ is the angular coordinate) with $z=f(re^{it})=\tan(r/2)e^{it}\in\mathbb C$. Then, for any function $v$ defined on $\Omega$ we have
$$
\int_{\Omega}v=\int_{f(\Omega)}(v\circ f^{-1})\frac{4}{(1+|z|^2)^2}dz.
$$
Moreover,
$$
|\Omega|=\int_{f(\Omega)}\frac{4}{(1+|z|^2)^2}dz\,,\ \ \  |\partial \Omega|=\int_{f(\partial\Omega)}\frac{2}{1+|z|^2}dz.
$$
Let $u$ be an eigenfunction associated to $\lambda_1(B(x_0,R),A_{x_0,\nu})$. As proved in Theorem \ref{technical_sphere_0}, we know that $u=u(r)$ is real, radial, and can be chosen such that $u,u'>0$ on $(0,R)$. In fact, from the our assumptions we have $R\leq\frac{\pi}{2}$, thus Theorem \ref{technical_sphere_0} applies. Then, $u\circ f^{-1}$ is positive and increasing as well.

\medskip

Let now $g:f(B(x_0,R))\rightarrow f(\Omega)$ be a conformal map with $f(0)=0$. To simplify our notation, we will set $\tilde\Omega=f(\Omega)$ and $\tilde B=f(B(x_0,R))$. Note that $\tilde B$ is a ball centered at $0$ of radius $T:=\tan(R/2)$.

\medskip

We set
$$
\hat u:=u\circ f^{-1}\circ g^{-1}\circ f
$$
Then $\hat u$ is a function defined in $\Omega$, the conformal transplantation of $u$ through the map $f^{-1}\circ g^{-1}\circ f$. The conformal invariance of the Aharonov-Bohm energy, proved in Appendix \ref{sec:conformal} tells that
$$
\int_{\Omega}|\nabla^{A_{x_0,\nu}}\hat u|^2=\int_{B(x_0,R)}|\nabla^{A_{x_0,\nu}}u|^2,
$$
then $\hat u\in H^1_{A_{x_0,\nu}}(\Omega)$ and it is a suitable test function for the min-max principle \eqref{minmax_N} for $\lambda_1(\Omega,A_{x_0,\nu})$. In particular
\begin{equation}\label{minmax_szego}
\lambda_1(\Omega,A_{x_0,\nu})\leq\frac{\int_{\Omega}|\nabla^{A_{x_0,\nu}}\hat u|^2}{\int_{\Omega}{\hat u}^2}=\frac{\int_{B(x_0,R)}|\nabla^{A_{x_0,\nu}} u|^2}{\int_{\Omega}{\hat u}^2}.
\end{equation}
If
\begin{equation}\label{claim_szego}
\int_{\Omega}{\hat u}^2\geq\int_{B(x_0,R)}u^2
\end{equation}
then we conclude from \eqref{minmax_szego}
$$
\lambda_1(\Omega,A_{x_0,\nu})\leq\frac{\int_{B(x_0,R)}|\nabla^{A_{x_0,\nu}} u|^2}{\int_{B(x_0,r)}u^2}=\lambda_1(B(x_0,R),A_{x_0,\nu}).
$$
which is what we want. Note that \eqref{claim_szego} is equivalent to
\begin{equation}\label{claim_szego_2}
\int_{\tilde B}U|g'|^2(\sigma\circ g)\geq\int_{\tilde B}U\sigma
\end{equation}
where
$$
U=(u\circ f^{-1})^2
$$
and $\sigma(z)=\frac{4}{(1+|z|^2)^2}$. We note that $U'>0$ on $(0,T)=(0,\tan(R/2))$. Let us set $B_r=B(0,r)$ the disk in $\mathbb C$ centered at $0$ of radius $r$. In particular $\tilde B=B_T=B_{\tan(R/2)}$. We define
$$
a(r):=\int_{B_r}|g'|^2(\sigma\circ g)\,,\ \ \ v(r)=\int_{B_r}\sigma=2\pi\int_0^r\frac{4r}{(1+r^2)^2}dr=\frac{4\pi r^2}{1+r^2}.
$$
Then, since $|\Omega|=|B(x_0,R)|$, we have $a(T)=v(T)$.

\medskip

We write a differential inequality for $a(r)$:
\begin{equation}\label{dineq_0}
a'(r)=\int_{\partial B_r}|g'|^2(\sigma\circ g)\geq\frac{\left(\int_{\partial B_r}|g'|(\sigma\circ g)^{1/2}\right)^2}{2\pi r}
\end{equation}
If $\Omega_r=(g\circ\sigma)^{-1}(B_r)\in\mathbb S^2$, then
$$
|\partial\Omega_r|=\int_{\partial B_r}|g'|(\sigma\circ g)^{1/2}\,,\ \ \ |\Omega_r|=\int_{B_r}|g'|^2(\sigma\circ g)=a(r)
$$
The isoperimetric inequality $|\partial\Omega_r|^2\geq|\Omega_r|(4\pi-|\Omega_r|)$ holds for spherical domains, hence from \eqref{dineq_0}
\begin{equation}\label{dineq}
a'(r)\geq\frac{a(r)(4\pi-a(r))}{2\pi r}
\end{equation}
Then, the function
$$
r\mapsto\frac{a(r)}{r^2(4\pi-a(r))}
$$
is not decreasing as long as $a(r)\leq 4\pi$, just take the derivative and use \eqref{dineq}. We see now that
$$
\frac{v(r)}{r^2(4\pi-v(r))}=1
$$
for all $r$, and in particular, since $v(T)=a(T)$,
$$
\frac{a(r)}{r^2(4\pi-a(r))}\leq \frac{a(T)}{T^2(4\pi-a(T))}=1
$$
and then
$$
a(r)\leq\frac{4\pi r^2}{1+r^2}=\int_{B_r}\sigma=v(r).
$$

This gives the desired result. In fact, since $U$ is radial with $U'>0$,
\begin{multline}\label{bandle_step1_sphere}
\int_{\tilde B}U|g'|^2(\sigma\circ g)=\int_0^{T}U(r)a'(r)dr\\
=U(T)a(T)-\int_0^{T}U'(r)a(r)
=U(T)v(T)-\int_0^{T}U'(r)a(r)\\
\geq U(T)v(T)-\int_0^{T}U'(r)v(r)=\int_{\tilde B}U\sigma.
\end{multline}
This proves \eqref{claim_szego_2} and then the isoperimetric inequality for $\lambda_1(\Omega,A_{x_0,\nu})$. Finally, if equality holds in the isoperimetric inequality for $\lambda_1(\Omega,A_{x_0,\nu})$, then equality holds in the isoperimetric inequality $|\partial\Omega_r|^2\geq|\Omega_r|(4\pi-|\Omega_r|)$ used in \eqref{dineq_0} for all $r$, hence all $\Omega_r$ are spherical disks, and $\Omega=\Omega_T$ as well.

\end{proof}

\section{The magnetic Steklov problem: Brock's and Weinstock's inequalities}\label{sec:brockweinstock}

 We now focus on the magnetic Steklov problem on a bounded smooth domain  $\Omega\subset\mathbb R^2$, namely problem \eqref{AB}. The min-max principle for the first eigenvalue reads
 \begin{equation}\label{minmax_S}
 \sigma_1(\Omega,A_{x_0,\nu})=\inf_{0\ne u\in H^1_{A_{x_0,\nu}}(\Omega,\mathbb C)}\frac{\int_{\Omega}|\nabla^{A_{x_0,\nu}}u|^2}{\int_{\partial\Omega}|u|^2},
 \end{equation}
where $H^1_{A_{x_0,\nu}}(\Omega,\mathbb C)$ denotes the standard magnetic Sobolev space (see Appendix \ref{functional} for the precise definition).

\smallskip

If we consider the maximisation problem for the lowest eigenvalue under volume constraint, we have Brock's Theorem for $\sigma_1(\Omega,A_{x_0,\nu})$:

\begin{thm}\label{thm:brock} Let $\Omega$ be a smooth bounded domain in $\mathbb  R^2$, $x_0\in\mathbb R^2$ a fixed pole, and let $B(x_0,r)$ be the disk with the same measure of $\Omega$. Let $\nu\in(0,\frac 12]$. Then:
$$
\sigma_1(\Omega,A_{x_0,\nu})\leq \sigma_1(B(x_0,R),A_{x_0,\nu})=
\dfrac{\sqrt{\pi}\nu}{\abs{\Omega}^{\frac 12}}.
$$
Equality holds if and only if $\Omega=B(x_0,r)$.
\end{thm}

\begin{proof}  From \eqref{minmax_S} we have:
$$
\sigma_1(\Omega,A_{x_0,\nu})\leq\frac{\int_{\Omega}|\nabla^Au|^2}{\int_{\partial\Omega}|u|^2}
$$
for all $u\in H^1_A(\Omega,\mathbb C)$. Then we choose $u=r^{\nu}$ which is the first eigenfunction for any disk centered at $x_0$ (see Appendix \ref{disk_S}). 
$$
|\nabla^A u|^2=2\nu^2r^{2\nu-2}.
$$
In particular, since $\nu\in\left(0,\frac{1}{2}\right]$,
$$
\int_{\Omega}|\nabla^A u|^2=2\nu^2\int_{\Omega}r^{2\nu-2}\leq 2\nu^2\int_{B(0,R)}r^{2\nu-2}=2\pi \nu R^{2\nu}.
$$
In fact, 
$$
\int_{\Omega\cap B(0,R)^c}r^{2\nu-2}\leq R^{2\nu-2}|\Omega\cap B(0,R)^c|=R^{2\nu-2}|\Omega^c\cap B(0,R)|\leq\int_{\Omega^c\cap B(0,R)}r^{2\nu-2}.
$$
Here $R=\frac{|\Omega|^{1/2}}{\pi^{1/2}}$ because $B(x_0,R)$ has the same volume of $\Omega$.

We recall a well-known fact: for all $p\geq 0$,
$$
\int_{\partial\Omega}r^p\geq 2\pi^{\frac{1-p}{2}}|\Omega|^{\frac{p+1}{2}}.
$$
When $p=0$ this is just the classical isoperimetric inequality. For $p>0$ this inequality says that the infimum of $\int_{\partial\Omega}r^p$ among all domains with fixed measure is attained by the ball centered at $x_0$, which is the unique minimizer. This result is proved in \cite{brock2}. Using $u$ as test function for $\sigma_1(\Omega,A_{x_0,\nu})$ and the isoperimetric inequality above with $p=2\nu$ we obtain
\begin{equation}\label{br1}
\sigma_1(\Omega,A_{x_0,\nu})\leq \frac{2\nu\pi R^{2\nu}}{2\pi^{\frac{1-2\nu}{2}}|\Omega|^{\frac{2\nu+1}{2}}},
\end{equation}
that is
\begin{equation}\label{br2}
|\Omega|^{\frac{1}{2}}\sigma_1(\Omega,A_{x_0,\nu})\leq\pi^{\frac{1}{2}}\nu.
\end{equation}
\end{proof}

By gauge invariance (see Appendix \ref{gauge_inv}), if $\nu\notin(0,\frac 12]$, we can replace $\nu$ in \eqref{br1} and \eqref{br2} by $\inf_{k\in\mathbb Z}|\nu-k|$.
\medskip

If we consider instead the problem of maximising the lowest eigenvalue under perimeter constraint, we have Weinstock's Theorem for $\sigma_1(\Omega,A_{x_0,\nu})$:


\begin{theorem}\label{thm:weinstock} Let $\Omega$ be bounded simply connected domain in $\mathbb  R^2$, $x_0\in\mathbb R^2$ be a fixed pole, and let $B(x_0,r)$ the disk with the same perimeter of $\Omega$. Let $\nu\in(0,\frac 12]$. Then:
$$
\sigma_1(\Omega,A_{x_0,\nu})\leq \sigma_1(B(x_0,R),A_{x_0,\nu})= \dfrac{2\pi}{\abs{\bd\Omega}}\nu.
$$
Equality holds if and only if $\Omega=B(x_0,R)$.
\end{theorem}
\begin{proof}
Assume for simplicity that $x_0=0$. Take the unique conformal map $\Phi:\Omega\to B$, where $B$ is the unit disk centered at the origin, and with $\Phi(0)=0$, and fix the eigenfunction $u=r^{\nu}$ of the unit disk, associated to $\sigma_1(B,A_{x_0,\nu})=\nu$. We refer to Appendix \ref{disk_S} for more details. We take as test-function
$$
\hat u=u\circ\Phi.
$$
Then:
$$
\sigma_1(\Omega,A)\int_{\bd\Omega}\abs{\hat u}^2=
\sigma_1(\Omega,\hat A)\int_{\bd\Omega}\abs{\hat u}^2
\leq\int_{\Omega}\abs{d^{\hat A}\hat u}^2
=\int_D\abs{d^Au}^2
=\sigma_1(D,A)\int_{\bd D}\abs{u}^2
=2\pi\nu
$$
where, on the first line, we used gauge invariance (Lemma \ref{gaugeinvariance})  and in the third we used the conformal invariance of the magnetic energy (Lemma \ref{cienergy}). Here $\hat A=\Phi^{\star}A$. On the other hand, $\hat u=1$ on $\bd\Omega$ so that 
$$
\int_{\bd\Omega}\abs{\hat u}^2=\abs{\bd\Omega}
$$
The conclusion is
$$
\sigma_1(\Omega,A_{x_0,\nu})\leq \dfrac{2\pi\nu}{\abs{\bd\Omega}}
$$
as asserted.
\end{proof}

\appendix


\section{Setting}\label{functional}

In this section we prove that problems \eqref{AB_N} and \eqref{AB} admit  purely discrete spectrum made of non-negative eigenvalues of finite multiplicity diverging to $+\infty$. Through all this section $\Omega$ will denote a bounded domain in $\mathbb R^2$. We will also prove the gauge invariance property of the two problems.

\subsection{Functional setting}

 Let $H^1(\Omega,\mathbb C)$ be the standard Sobolev space of complex-valued functions. As for problem \eqref{AB_N}, we shall assume that $\Omega$ is such that the embedding $H^1(\Omega,\mathbb C)\subset L^2(\Omega,\mathbb C)$ is compact. As for problem \eqref{AB} we shall assume that the trace operator $\gamma_0:H^1(\Omega,\mathbb C)\rightarrow L^2(\partial\Omega,\mathbb C)$ is compact. In both cases the compactness is guaranteed if $\Omega$ is smooth (Lipschitz is enough).

Let $x_0\in\mathbb R^2$ and let $C^{\infty}_{x_0}(\Omega,\mathbb C)$ be the space of smooth functions on $\Omega$ vanishing in a neighborhood of $x_0$. We introduce the magnetic Sobolev space $H^1_A(\Omega,\mathbb C)$ defined as the closure of $\{u\in C^{\infty}_{x_0}(\Omega,\mathbb C):\nabla^Au,u\in L^2(\Omega,\mathbb C)\}$ with respect to the norm
$$
\|u\|_A^2:=\int_{\Omega}|\nabla^Au|^2 + |u|^2 \,,\ \ \ \forall u\in C^{\infty}_0(\Omega,\mathbb C):\nabla^Au,u\in L^2(\Omega,\mathbb C).
$$ 
To simplify the notation, through all this section we shall denote by $A$ the Aharonov-Bohm potential $\nu A_0$ with pole at $x_0$.
We will make use of equivalent norms, suitable to our problems. One fundamental equivalent norm is given in the following proposition.

\begin{prop}\label{eq1}
If $\nu\notin\mathbb Z$, the norm $\|u\|_A$ is equivalent to the following norm:
$$
\|u\|_{A'}^2:=\int_{\Omega}|\nabla u|^2 + |u|^2 +\frac{|u|^2}{|x-x_0|^2}.
$$
\end{prop}
If $\Omega$ does not contain the pole $x_0$, then the equivalence is immediate to check. In any case, the proof of Proposition \ref{eq1} is standard, and follows from the well-known Hardy-type inequality proved in \cite{laptev_weidl}:

\begin{lemme}\label{h}
For any $R>0$ and any $u\in H^1(B(x_0,R),\mathbb C)$ we have
\begin{equation}\label{hardy}
\int_{B(x_0,R)}|\nabla^A u|^2dx\geq C^2\int_{B(x_0,R)}\frac{|u|^2}{|x-x_0|^2},
\end{equation}
where $C=\inf\{|\nu-k|:k\in\mathbb Z\}$. In particular, if $\nu\in\left(0,\frac{1}{2}\right]$, then $C=\nu$.
\end{lemme}
\begin{proof}
Inequality \eqref{hardy} is proved in \cite{laptev_weidl}. We recall briefly here the proof. We use polar coordinates $(r,\theta)$ in $\mathbb R^2$ centered at $x_0$, and express $\int_{B(x_0,R)}|\nabla^Au|^2dx$ in this new coordinate system. We have
$$
\int_{B(x_0,R)}|\nabla^A u|^2dx=\int_0^R\int_0^{2\pi}\left(|\partial_r u|^2+\frac{1}{r^2}|\partial_{\theta}u-i\nu  u|^2\right)rd\theta dr.
$$

We focus on $\int_0^{2\pi}|\partial_{\theta}u-i\nu  u|^2d\theta$. We estimate 
$$
\inf_{0\ne u\in H^1(\mathbb S^1,\mathbb C)}\frac{\int_0^{2\pi}|\partial_{\theta}u-i\nu u|^2d\theta}{\int_0^{2\pi}r^2|u|^2d\theta}.
$$
This infimum corresponds to the first eigenvalue, which we denote by $\mu_1$, of
$$
-\partial^2_{\theta\theta}u+2i\nu \partial_{\theta}u+\nu^2 u=\mu u
$$
on $\mathbb S^1$. A set of $L^2(\mathbb S^1,\mathbb C)$-normalized eigenfunctions is given by $(2\pi)^{-1/2}e^{i\theta k}$, $k\in\mathbb Z$, with corresponding eigenvalue $\mu=(k-\nu)^2$. If $\nu\notin\mathbb Z$, then $\mu_1=\inf\{(k-\nu)^2:k\in\mathbb Z\}$.

This concludes the proof.
\end{proof}

We recall the following well-known result
\begin{lemme}\label{bdry_norm}
Let $\Omega$ be a bounded domain in $\mathbb R^n$ with Lipschitz boundary. Then the norm $\|u\|_{H^1(\Omega,\mathbb C)}^2$ is equivalent to the following norm:
\begin{equation}\label{normS}
\|u\|_{\partial}^2:=\int_{\Omega}|\nabla u|^2+\int_{\partial\Omega}|u|^2,
\end{equation}
for all $u\in H^1(\Omega,\mathbb C)$. With abuse of notation, we write $\int_{\partial\Omega}|u|^2$ in place of $\int_{\partial\Omega}|\gamma_0(u)|^2$. 
\end{lemme}

A consequence of Lemmas \ref{eq1} and \ref{bdry_norm} is the following

\begin{lemme}\label{eq2}
For $\nu\notin\mathbb Z$, the norm $\|u\|_A$ is equivalent to the following norms:
\begin{equation}\label{Abdry1}
\|u\|^2_{A,\partial}:=\int_{\Omega}|\nabla^Au|^2+\int_{\partial\Omega}|u|^2,
\end{equation}
and
\begin{equation}\label{Abdry2}
\|u\|^2_{A',\partial}:=\int_{\Omega}|\nabla u|^2+\frac{|u|^2}{|x-x_0|^2}+\int_{\partial\Omega}|u|^2.
\end{equation}
\end{lemme}

\begin{rem}
Lemma \ref{eq1} essentially says that, if $\nu\notin\mathbb Z$
$$
H^1_A(\Omega,\mathbb C)=\left\{u\in H^1(\Omega,\mathbb C):\frac{u}{|x-x_0|}\in L^2(\Omega,\mathbb C)\right\}.
$$
In particular, $H^1_A(\Omega,\mathbb C)\subset H^1(\Omega,\mathbb C)$.
\end{rem}

\subsection{The Neumann eigenvalue problem}

Problem \eqref{AB_N} is understood in the weak sense, namely, find a function $u\in H^1_A(\Omega,\mathbb C)$ and number $\lambda\in\mathbb C$ such that
$$
\int_{\Omega}\langle\nabla^A u,\overline{\nabla^A \phi}\rangle=\lambda\int_{\Omega}u\bar\phi \,,\ \ \ \forall\phi\in H^1_A(\Omega,\mathbb C).
$$
We rewrite it as
\begin{equation}\label{weak_N}
\int_{\Omega}\langle\nabla^A u,\overline{\nabla^A \phi}\rangle  +\int_{\Omega}u\bar\phi =(\lambda+1)\int_{\Omega}u\bar\phi \,,\ \ \ \forall\phi\in H^1_A(\Omega,\mathbb C).
\end{equation}

Since the quadratic form on the left-hand side of \eqref{weak_N} is bounded and coercive on $H^1_A(\Omega,\mathbb C)$  and the embedding $H^1_A(\Omega,\mathbb C)\subset L^2(\Omega,\mathbb C)$ is compact, problem \eqref{weak_N} is recast to an eigenvalue problem for a compact self-adjoint operator on the Hilbert space $H^1_A(\Omega,\mathbb C)$. Standard Spectral Theory implies that problem \eqref{weak_N} admits a sequence of eigenvalues
$$
-\infty<\lambda_1\leq\lambda_2\leq\cdots\leq\lambda_k\leq\cdots\nearrow+\infty.
$$

Moreover, there exists a Hilbert basis of $H^1_A(\Omega,\mathbb C)$ of eigenfunctions of \eqref{weak_N}. The eigenfunctions can be chosen to form a orthonormal basis of $L^2(\Omega,\mathbb C)$ as well.

\smallskip

We finally recall the variational characterization of the eigenvalues:

\begin{equation}\label{minmax_N}
\lambda_k=\min_{\substack{U\subset H^1_A(\Omega,\mathbb C)\\{\rm dim}\,U=k}}\max_{0\ne u\in U}\frac{\int_{\Omega}|\nabla^Au|^2}{\int_{\Omega}|u|^2}.
\end{equation}
In particular, $\lambda_1\geq 0$. If $\nu\notin\mathbb Z$ and $x_0\in\Omega$, we deduce from Lemma \ref{h} that $\lambda_1>0$. On the other hand, if $\nu\in\mathbb Z$, then $e^{i\nu\theta}\in H^1_A(\Omega,\mathbb C)$ is an eigenfunction corresponding to the eigenvalue $\lambda=0$. We deduce by gauge invariance (see Appendix \ref{gauge_inv} ) that in this case the eigenvalues of \eqref{AB_N} coincide with those of the Neumann Laplacian (i.e., $\nu=0$). If $x_0$ belongs to the unbounded component of $\Omega^c$, then the flux of $A_{x_0,\nu}$ is zero, hence $A_{x_0,\nu}=\nabla\phi$ with $\Delta\phi=0$ on $\Omega$, and the eigenvalues of \eqref{AB_N} coincide with those of the Neumann Laplacian. If $x_0$ belongs to some bounded component of $\Omega^c$, the flux is $\nu$, and if it is not an integer, inequality \eqref{hardy} with $C>0$ holds with $B(x_0,R)$ replaced by an annular region $A\subset\Omega$ such that $x_0$ belongs to the bounded component of $A^c$ (one immediately realizes that \eqref{hardy} holds if we replace $B(x_0,R)$ by $B(x_0,R)\setminus B(x_0,r)$ for any $0<r<R$). Then $\lambda_1>0$.

\subsection{The Steklov eigenvalue problem}

Problem \eqref{AB} is understood in the weak sense, namely, find a function $u\in H^1_A(\Omega,\mathbb C)$ and number $\sigma\in\mathbb C$ such that
$$
\int_{\Omega}\langle\nabla^A u,\overline{\nabla^A \phi}\rangle=\sigma\int_{\partial\Omega}u\bar\phi\,,\ \ \ \forall\phi\in H^1_A(\Omega,\mathbb C).
$$
We rewrite it as
\begin{equation}\label{weak}
\int_{\Omega}\langle\nabla^A u,\overline{\nabla^A \phi}\rangle  +\int_{\partial\Omega}u\bar\phi =(\sigma+1)\int_{\partial\Omega}u\bar\phi \,,\ \ \ \forall\phi\in H^1_A(\Omega,\mathbb C).
\end{equation}

For the study of problem \eqref{weak} we consider the space $H^1_A(\Omega,\mathbb C)$ endowed with the equivalent norm $\|u\|_{A,\partial}$.

Since the quadratic form on the left-hand side of \eqref{weak} is bounded and coercive on $H^1_A(\Omega,\mathbb C)$ (with the equivalent norm $\|u\|_{A,\partial}$, see Lemma \ref{eq2}) and the trace operator $\gamma_0:H^1_A(\Omega,\mathbb C)\rightarrow L^2(\partial\Omega,\mathbb C)$ is compact, problem \eqref{weak} is recast to an eigenvalue problem for a compact self-adjoint operator on the Hilbert space $H^1_A(\Omega,\mathbb C)$. Standard Spectral Theory implies that problem \eqref{weak} admits a sequence of eigenvalues
$$
-\infty<\sigma_1\leq\sigma_2\leq\cdots\leq\sigma_k\leq\cdots\nearrow+\infty.
$$

Moreover, there exists a Hilbert basis of $H^1_{A,h}(\Omega,\mathbb C)$ of eigenfunctions of \eqref{weak}, where 
$$
H^1_{A,h}(\Omega,\mathbb C)=\left\{u\in H^1_A(\Omega,\mathbb C):\int_{\Omega}\langle\nabla^Au,\overline{\nabla^A\phi}\rangle=0\,,\ \ \ \forall \phi\in H^1_A(\Omega,\mathbb C):\gamma_0(u)=0\right\}.
$$
The eigenfunctions can be chosen to form a orthonormal basis of $L^2(\partial\Omega,\mathbb C)$ as well.

\smallskip

We finally recall the variational characterization of the eigenvalues:

\begin{equation}\label{minmax}
\sigma_k=\min_{\substack{U\subset H^1_A(\Omega,\mathbb C)\\{\rm dim}\,U=k}}\max_{0\ne u\in U}\frac{\int_{\Omega}|\nabla^Au|^2}{\int_{\partial\Omega}|u|^2}.
\end{equation}
In particular, $\sigma_1\geq 0$. As in the Neumann case, when $x_0\in\Omega$, $\sigma_1>0$ if and only if $\nu\notin\mathbb Z$. If $\nu\in\mathbb Z$ or $x_0$ belongs to the unbounded connected component of $\Omega^c$, the eigenvalues coincide with those of the usual Steklov problem (i.e., $\nu=0$). If $\nu\notin\mathbb Z$ and $x_0$ belongs to some bounded connected component of $\Omega^c$, then $\sigma_1>0$.

\smallskip

We finally remark that the discussion contained in this section applies with no essential modifications to the case of bounded domains in two-dimensional Riemannian manifolds  $M^2$, and also to the case in which multiple singularities (of Aharonov-Bohm type) occur, i.e., the domain is punctured in many points.

\subsection{Gauge invariance}\label{gauge_inv}

It is well-known that, if $A_{x_0,\nu}$ and $A_{\nu'}$ have fluxes $\nu,\nu'$ which differ by an integer, i.e., $\nu-\nu'\in\mathbb Z$, then $\Delta_{A_{x_0,\nu}}$ and $\Delta_{A_{\nu'}}$ are unitarily equivalent. In fact, one can define a multivalued function
$$
\psi(x)=\int_{y_0}^x(A_{x_0,\nu}-A_{\nu'})
$$
where $y_0$ is any reference point (but not the pole). By the given condition, $\psi$ is multi valued but $e^{i\psi}$ is well-defined. Gauge invariance identity is:
$$
\Delta_{A_{\nu'}}e^{-i\psi}=\Delta_{A_{x_0,\nu}-\nabla\psi}e^{i\psi}=e^{-i\psi}\Delta_{A_{x_0,\nu}}.
$$
One also observes that, if
$$
\nabla^{A_{x_0,\nu}}u=\nabla u-iuA_{x_0,\nu},
$$
then one has the identity:
$$
\nabla^{A_{x_0,\nu}-\nabla\psi}(e^{-i\psi}u)=e^{-i\psi}\nabla^{A_{x_0,\nu}}u.
$$

From this one gets easily that, if $u$ is an eigenfunction of the magnetic Neumann problem \eqref{AB_N} associated to $\lambda$, then $e^{-i\psi}u$ is an eigenfunction of the corresponding problem with potential $A_{x_0,\nu}-\nabla\psi$, associated to the same eigenvalue. The same assertion holds for the magnetic Steklov problem \eqref{AB}.

\smallskip

Since we are interested in the case $\nu\notin\mathbb Z$, we can always assume, without loss of generality, that 
\begin{equation}\label{gauge_ass}
\nu\in \left(0,\frac 12\right].
\end{equation}

\section{Aharonov-Bohm eigenvalues in two-dimensional manifolds of revolution}\label{sec:AB_revolution}

Let $(M^2,g)$ be a $2$-dimensional compact manifold of revolution with pole $x_0$ and diameter $D$. We assume that the metric, in polar coordinates $(r,t)$ around the pole, is given by:
$$
g=\twomatrix 100{\theta^2 (r)},
$$
with $\theta(r)$ smooth and positive. Here $r$ is the distance to $x_0$ (i.e., the radial coordinate). We shall always assume that
$$
\theta>0 {\rm\ on\ }(0,D)\,,\ \ \  \theta'(0)=1.
$$
We also define
 $$
 \bar R=\min\{R:\theta'(R)=0\}\in(0,D).
 $$
If such minimum does not exists, then we set $\bar R=D$.

\medskip

If $f=f(r,t)$ the Laplacian writes:
$$
\Delta f=-f''-\dfrac{\theta'}{\theta}f'-\frac{1}{\theta^2}\deriven 2f t,
$$
where by $f'$ we denote the derivative of $f$ with respect to the radial variable $r$.
\smallskip

We consider the harmonic $1$-form $A_{x_0,\nu}$ which is written in polar coordinates as
$$
A=\nu dt
$$
where $\nu$ is the flux around $x_0$.

\begin{lemme} For any $f=f(r,t)$:
$$
\Delta_{A_{x_0,\nu}}f=-f''-\dfrac{\theta'}{\theta}f'-\dfrac{1}{\theta^2}\deriven 2f{t}+\dfrac{\nu^2}{\theta^2}f+2i\dfrac{\nu}{\theta^2}\derive f{t}.
$$
If $f(r,t)=u(r)e^{ikt}$ then 
$$
\derive f{t} =iku e^{ikt}, \quad \deriven 2 f{t}=-k^2 u e^{ikt}
$$
and therefore:
$$
\Delta_{A_{x_0,\nu}}f=\left(-u''-\dfrac{\theta'}{\theta}u'+\dfrac{(k-\nu)^2}{\theta^2}u\right)e^{ikt}.
$$

\end{lemme}
\begin{proof}

The magnetic Laplacian writes, for co-closed potentials $A_{x_0,\nu}$:
$$
\Delta_{A_{x_0,\nu}}f=\Delta f+\abs{A_{x_0,\nu}}^2f+2i\scal{df}{A_{x_0,\nu}}.
$$

Normalizing the basis $\left(\derive{}{r}, \derive {}{t}\right)$ we obtain the orthonormal basis:
$$
(e_1,e_2)=\left(\derive{}{r}, \frac {1}{\theta}\derive {}{t}\right)
$$
hence we get:
$$
\abs{A_{x_0,\nu}}^2=\dfrac{\nu^2}{\theta^2}
$$
and 
$$
2i\scal{df}{A_{x_0,\nu}}=2i\dfrac{\nu}{\theta^2}\derive f{t}.
$$
The claim immediately follows.
\end{proof}

\begin{cor}
If $f(r,t)=u(r)e^{ikt}$, then  $\Delta_{A_{x_0,\nu}}f=\lambda f$ reads:
\begin{equation}\label{radial_part}
u''+\frac {\theta'}{\theta} u'+\left(\lambda-\frac{(k-\nu)^2}{\theta^2}\right)u=0.
\end{equation}
\end{cor}

\subsection{Eigenfunctions of the Neumann problem}\label{sub:rev_N}
We consider now problem \eqref{AB_N} on disks $B(x_0,R)$ of radius $R$ centered at the pole of a manifold of revolution $M^2$.

\smallskip

Looking for eigenfunctions of the form $f(r,t)=u(r)e^{ikt}$, $k\in\mathbb Z$, we find, thanks to \eqref{radial_part}, that $u$ is a bounded solution of the following singular Sturm-Liouville problem.
\begin{equation}\label{SL_N}
\begin{cases}
u''+\frac {\theta'}{\theta} u'+\left(\lambda-\frac{(k-\nu)^2}{\theta^2}\right)u=0 & {\rm in\ }(0,R),\\
u'(R)=0.
\end{cases}
\end{equation}
If $\nu\notin\mathbb Z$, we see from \eqref{SL_N} that necessarily $u(0)=0$  (and in fact, any function in $H^1_A(\Omega,\mathbb C)$ vanishes at $x_0$). If $\nu\in\mathbb Z$, we see that any bounded solution of \eqref{SL_N} with $\nu=k$ satisfies $u'(0)=0$, while for $\nu\ne k$, $u(0)=0$. This is exactly the case of the standard Laplacian. The condition $u'(R)=0$ is the Neumann condition; in fact, $\langle \nabla^{A_{x_0,\nu}}u,N\rangle=\langle\nabla u,N\rangle=u'(R)e^{ikt}$ in the case of a disk centered at $x_0$.

\smallskip

The associated quadratic form is 
$$
\int_0^R \left(u'^2+ V_ku^2\right)\theta\,dr
$$
where $V_k(r)=\frac{(k-\nu)^2}{\theta^2(r)}$. Note that, if $\nu\in (0,\frac 12]$, then $V_k$ is increasing in $\abs{k}$,  hence:
$$
V_k\geq V_0,
$$
for all $k$.

\begin{lemma}\label{lemma_N} Let $\nu\in (0,\frac 12]$. Then:
\begin{enumerate}[i)]

\item For each $k\in\mathbb Z$, problem \eqref{SL_N} admits an infinite sequence of eigenvalues:
$$
0<\lambda_{k1}<\lambda_{k2}<\cdots\leq\lambda_{kj}<\cdots\nearrow+\infty
$$
with associated bounded eigenfunctions $u_{kj}$. All eigenvalues are simple and $u_{kj}$ has $j$ zeros in $(0,R)$.

\item If $\abs{k}\leq \abs{h}$ then:
$$
\lambda_{k1}\leq \lambda_{h1}.
$$
In particular, $\min_{k,j}\{\lambda_{kj}\}=\lambda_{01}$.
\end{enumerate}
\end{lemma}
\begin{proof}
Point $i)$ follows from standard theory for singular Sturm-Liouville problems, since $\theta(r)$ is smooth and positive in $(0,R]$, and $\theta(r)\sim \theta'(0) r$ as $r\rightarrow 0^+$.

\smallskip

To prove $ii)$ it is sufficient to consider $k=0$. The claim is a consequence of the fact that $V_h$ is increasing in $h$. Let us take $u=u_{h1}$, an eigenfunction associated to $\lambda_{h1}$, as test function for $\lambda_{01}$ in its variational characterization (which is the analogue of \eqref{minmax_V} weighted with $\theta$). We get
$$
\lambda_{01}\int_0^Ru^2\theta\leq \int_0^R \left(u'^2+ V_0u^2\right)\theta\,dr
\leq  \int_0^R \left(u'^2+ V_hu^2\right) \theta\,dr
=\lambda_{h1}\int_0^Ru^2\theta dr
$$
\end{proof}

\begin{cor}\label{cor_first}
We have
$$
\lambda_1(B(x_0,R),A_{x_0,\nu})=\lambda_{01}.
$$
A first eigenfunction is given by $u_{01}(r)$, and it is real and radial.
\end{cor}

Note that $\lambda_1(B(x_0,R),A_{x_0,\nu})$ is not necessarily simple (this is the case of $\nu=\frac{1}{2}$ and the unit disk, see Appendix \ref{disk}). At any rate, Corollary \ref{cor_first} states that we can always find a first eigenfunction which is real and radial. When $\nu\in\mathbb Z$, Lemma \ref{lemma_N} gives exactly the eigenvalues of the Neumann Laplacian, and in particular $\lambda_{01}=0$.

\smallskip

We have the following result on completeness of the family of eigenfunctions $\{u_{kj}(r)e^{ikt}\}$.

\begin{theorem}\label{complete} Let 
$$
\psi_{kj}(r,t)=u_{kj}(r)e^{ikt}
$$
for all $k\in\mathbb Z$, $j=1,2,...$, where $u_{kj}$ are the eigenfunctions of \eqref{SL_N}. Then $\{\psi_{kj}(r,t)\}_{k,j}$ is a complete system in $L^2(B(x_0,R))$.

\end{theorem}

\begin{proof} If $\nu\in\mathbb Z$ the result is well-known: in fact the eigenfunctions $\{\psi_{kj}\}$ are the eigenfunctions of the Neumann Laplacian. Assume then $\nu\notin\mathbb Z$. Let us assume that the system of eigenfunctions $\{\psi_{kj}(r,t)\}_{k,j}$ is not complete. Then, there exists an eigenfunction $f$ of \eqref{AB_N} which is orthogonal to the span of all $\psi_{kj}$, that is:
\begin{equation}\label{ortho}
\int_{B(x_0,R)} f\overline\psi_{hi}=0
\end{equation}
for all $h\in{\mathbb Z}, i=1,2\dots$. Expand $f(r,t)$ in a Fourier series, for every fixed $r$, and get:
$$
f(r,t)=\sum_{k\in{\mathbb Z}}a_k(r)e^{ikt}
$$
As by assumption $f\ne 0$, there exists $h\in {\mathbb Z}$ such that $a_h\ne 0$.

We prove in Lemma \ref{lemma_ak} here below that, for all $k\in\mathbb Z$,
\begin{equation}\label{claim}
a_k(r)=c_k u_{k{j_k}}(r)
\end{equation}
for some $j_k=1,2,...$ and $c_k\in\mathbb C$, $c_k\ne 0$.
The conclusion follows from \eqref{claim}. In fact, we can write
$$
f(r,t)= \sum_{k\in{\mathbb Z}}c_ku_{k,j_k}(r)e^{ikt}=\sum_{k\in{\mathbb Z}}c_k\psi_{k,j_k}(r,t)
$$
and then, by the orthogonality of $\psi_{kj}$'s:
\begin{equation}
\int_{B(x_0,R)}f\overline\psi_{h{j_h}}=c_h\ne 0
\end{equation}
which contradicts \eqref{ortho}. 
\end{proof}
\begin{lemma}\label{lemma_ak}
Let $\{u_{kj}\}$ denote the eigenfunctions of \eqref{SL_N} and assume that $f(r,t)=\sum_{k\in\mathbb Z}a_k(r)e^{ikt}$ is a nonzero eigenfunction of \eqref{AB_N}. Then, for all $k\in\mathbb Z$
$$
a_k(r)=c_k u_{k{j_k}}(r)
$$
for some $j_k=1,2,...$.
\end{lemma}
\begin{proof}
We show that for each $k\in\mathbb Z$ the function $a_k:[0,R]\to\reals$ is a bounded solution of
$$
\begin{cases}
a_k''+\frac {\theta'}{\theta} a_k'+\left(\lambda-\frac{(k-\nu)^2}{\theta^2}\right)a_k=0 & {\rm in\ } (0,R)\\
a_k'(R)=0.
\end{cases}
$$
In particular, since $\nu\notin\mathbb Z$, $a_k(0)=0$.
Hence, since the eigenvalues are simple, for every $k$, there exists $j=j_k$ and $c_k\in {\mathbb C}$ such that
$$
a_k(r)=c_ku_{k,j_k}(r),\quad \lambda=\lambda_{k,j_k}.
$$
If we assume $a_k\ne 0$, then $c_k\ne 0$.

First observe that $f(x_0)=0$ because otherwise $f$ is not in $H^1_A(\Omega,\mathbb C)$ when $\nu\notin\mathbb Z$. Then:
$$
0=f(x_0)=\sum_{k\in{\mathbb Z}}a_k(0)e^{ikt}
$$
for all $t$. Hence, by multiplying by $e^{-iht}$ and integrating on $[0,2\pi]$ we see that $a_h(0)=0$. Then $a_k(0)=0$ for all $k$. Similarly, the Neumann condition 
$\derive fr(R,t)=0$ implies $a'_k(R)=0$ for all $k$. 

\smallskip

It remains to show that $a_k$ is an eigenfunction of the problem \eqref{SL_N}.
First observe that
$$
a_k(r)=\int_0^{2\pi}f(r,t)e^{-ikt}\,dt
$$
Hence:
$$
-k^2a_k=\int_0^{2\pi}f(r,t)\deriven 2{}{t}e^{-ikt}\,dt
=\int_0^{2\pi}\deriven 2{f}{t}(r,t)e^{-ikt}\,dt
$$
therefore
\begin{equation}\label{principal}
-\frac{k^2}{\theta^2}a_k=\int_0^{2\pi}\frac{1}{\theta^2}\deriven 2{f}{t}(r,t)e^{-ikt}\,dt
\end{equation}
Now we know that $\Delta_{A_{x_0,\nu}}f=\lambda f$; hence
$$
\dfrac{1}{\theta^2}\deriven 2f{t}=-\lambda f-f''-\dfrac{\theta'}{\theta}f'+\dfrac{\nu^2}{\theta^2}f+2i\dfrac{\nu}{\theta^2}\derive f{t}.
$$
so that
\begin{multline*}
\int_0^{2\pi}\frac{1}{\theta^2}\deriven 2{f}{t}e^{-ikt}\,dt
=-\lambda\int_0^{2\pi}fe^{-ikt}dt-\int_0^{2\pi}f''e^{-ikt}dt\\
-\dfrac{\theta'}{\theta}\int_0^{2\pi}f'e^{-ikt}dt+\dfrac{\nu^2}{\theta^2}\int_0^{2\pi}fe^{-ikt}dt+2i\dfrac{\nu}{\theta^2}\int_0^{2\pi}\derive f{t}e^{-ikt}\,dt
\end{multline*}
Integrating by parts:
$$
\int_0^{2\pi}\derive f{t}e^{-ikt}\,dt=ik\int_0^{2\pi}fe^{-ikt}dt=ika_k.
$$
Then:
$$
\int_0^{2\pi}\frac{1}{\theta^2}\deriven 2{f}{t}e^{-ikt}\,dt=-\lambda a_k-a_k''-\frac{\theta'}{\theta}a_k'+\frac{\nu^2}{\theta^2}a_k-2k\frac{\nu}{\theta^2}
$$
which, substituted in \eqref{principal}, gives:
$$
a_k''+\frac{\theta'}{\theta}a_k'+\left(\lambda-\frac{(k-\nu)^2}{\theta^2}\right)a_k=0
$$
which is the final assertion. 

\end{proof}

In the following theorems we collect a few useful properties of the eigenfunction $u=u_{01}$ associated with the first eigenvalue $\lambda=\lambda_{01}$.

\begin{thm}\label{technical_sphere_0}
Let $\nu\in(0,\frac{1}{2}]$. Let $\lambda>0$ denote the first eigenvalue of problem \eqref{SL_N} with $k=0$ and $R\in(0,\bar R]$, and let $u$ denote a corresponding first eigenfunction.  Then the following statements hold:
\begin{enumerate}[i)]
\item $u(0)=0$ and $u\ne 0$ on $(0,R)$. In particular, we can choose $u>0$ on $(0,R)$.
\item $u'>0$ on $(0,R)$ and $\lambda>\frac{\nu^2}{\theta(R)^2}$.
\item The first eigenvalue $\lambda$ is strictly decreasing in $R$ for $R\in(0,\bar R)$. In particular $\lambda>\bar\lambda$, where $\bar\lambda$ is the first eigenvalue of \eqref{SL_N} with $R=\bar R$.
\end{enumerate}
\end{thm}

\begin{proof}
Point $i)$ has been already discussed in Lemma \ref{lemma_N} and follows from standard Sturm-Liouville theory and the fact that $\theta,\theta'>0$ on $(0,R)$, $\theta(r)\sim \theta'(0)r$ as $r\rightarrow 0^+$. 

\medskip

Consider now $ii)$. Set $N(r):=\theta(r)u'(r)$. Then $N(0)=N(R)=0$ and
$$
N'(r)=\left(\frac{\nu^2}{\theta^2(r)}-\lambda\right)\theta(r)u(r).
$$
Since $N'(r)>0$ on $(0,\delta)$ for some $\delta>0$, we deduce that the function $N$ increases from zero and decreases to zero. It has only one maximum point in $(0,R)$. In fact, since $\theta$ and $u$ are strictly positive in $(0,R)$, $N'(r)=0$ if and only if $\lambda=\frac{\nu^2}{\theta^2(r)}$. Since $\frac{1}{\theta^2}$ is strictly decreasing in $(0,R)$, it vanishes exactly once in $(0,R)$. This implies that $N>0$ on $(0,R)$ and therefore $u'>0$ on $(0,R)$. Moreover $N'(R)< 0$.

\medskip

As for $iii)$, let us write $\lambda=\lambda(R)$ to highlight the dependence on $R$. Then, as $\lambda(R)$ is simple, we derive the identity
$$
\int_0^R u^2\theta\lambda(R)=\int_0^R\left(u'^2+\frac{\nu^2}{\theta^2}u^2\right)\theta
$$
and obtain
$$
\lambda'(R)=\left(\frac{\nu^2}{\theta^2(R)}-\lambda(R)\right)\frac{u^2(R)\theta(R)}{\int_0^Ru^2\theta}, 
$$
which is strictly negative by point $ii)$ as long as $\theta'< 0$, that is, $R\in(0,\bar R)$. More details can be found e.g., in \cite{dauge_helffer}.
\end{proof}

\begin{thm}\label{technical_sphere}
Let $\nu\in(0,\frac{1}{2}]$. Let $\lambda>0$ denote the first eigenvalue of problem \eqref{SL_N} with $k=0$ and $R\in(0,\bar R)$, and let $u$ denote a corresponding eigenfunction. Let $F(r)=u'(r)^2+\frac{\nu^2u(r)^2}{\theta(r)^2}$. If either 
\begin{enumerate}[a)]
\item $\theta'(r)\geq \nu$ on $(0, R)$, or
\item $\bar \lambda\theta^2-\nu^2+\nu^2(\theta')^2+\nu\theta\theta''\geq 0$ on $(0, R)$,
\end{enumerate}
then $F'(r)\leq 0$ on $(0,R)$. If moreover $\theta^{\nu}$ is a solution to \eqref{SL_N} for $R=\bar R$, then  $F'(r)\leq 0$ on $(0,\bar R)$

\end{thm}

\begin{proof}
Define
$$
q(r)=\theta(r)\frac{u'(r)}{u(r)}.
$$
Then $F=u'^2+\frac{\nu^2 u}{\theta^2}=\frac{u^2}{\theta^2}(q^2+\nu^2)$,  and with standard computations we see that
$$
F'=2qq'\frac{u^2}{\theta^2}+\frac{2 u^2}{\theta^3}(q^2+\nu^2)(q-\theta').
$$
Since $u,\theta,\theta'>0$ on $(0,R)$, the claim follows provided $q'\leq 0$ and $q\leq\theta'$ on $(0,R)$. 

\medskip

We start by proving $q'\leq 0$ on $(0,R)$. We see that
$$
q'=-\theta\lambda+\frac{\nu^2}{\theta}-\frac{q^2}{\theta}.
$$
The local behavior of $u$ near $r=0^+$ depends only on the differential equation \eqref{SL_N}, and in particular, using the Frobenius-Taylor expansion, we deduce that $u\sim c r^{\nu}$ as $r\rightarrow 0^+$ for some $c>0$. Therefore $q(0)=\nu$, $q'(0)=0$. Clearly there exists $a>0$ such that $q\not\equiv\nu$ on $(0,a)$, otherwise $q'=-\theta\lambda\ne 0$ on $(0,a)$. Choose then (a possibly smaller) $a>0$ such that $q,q'$ have constant sign on $(0,a)$ (in particular, $q>0$ on $(0,R)$). We have $q'<0$  in $(0,a)$. Otherwise, $q>q(0)=\nu$, but then the differential equation tells then $q'<0$. Then $q$ starts decreasing. The same argument  shows that $0<q<\nu$ on $(0,R)$. If $q'$ changes sign in $(0,R)$, let then $b,c,d\in(0,R)$, $b<c<d$, be such that $q'(b)<0$, $q'(c)>0$, $q'(d)<0$ and $q(b)=q(c)=q(d)=\bar q<\nu$. Such points exist if $q'$ changes sign since $q(R)=0$. We note now that, since $\theta'>0$ in $(0,R)$, $\theta(c)=t\theta(b)+(1-t)\theta(d)$ for some $t\in(0,1)$. Then, from the differential equation for $q'$, we see that
\begin{multline*}
0<q'(c)=-\lambda\theta(c)+	\frac{\nu^2-\bar q^2}{\theta(c)}\\
\leq t\left(-\lambda\theta(b)+	\frac{\nu^2-\bar q^2}{\theta(b)}\right)+(1-t)\left(-\lambda\theta(d)+	\frac{\nu^2-\bar q^2}{\theta(d)}\right)=tq'(b)+(1-t)q'(d)<0,
\end{multline*}
a contradiction. Then $q'<0$ on $(0,R)$. Note that $q'<0$ also in the case $R=\bar R$.

\medskip

It remains to prove that $q\leq\theta'$. Note that, if $\theta'\geq\nu$ on $(0,R)$, since $q$ is decreasing and $q(0)=\nu$, we have $q\leq\theta'$. We have used condition $a)$. Note that this condition has already been found in Theorem \ref{schr}. 

\medskip Assume now $b)$ holds. Let $R\in(0,\bar R)$. In order to conclude, it is sufficient to prove
\begin{equation}\label{cond1}
q\leq (\nu+\delta)\theta'
\end{equation}
on $(0,R)$, for some $\delta\in(0,1-\nu)$.  For any $\delta>0$ we have $q(0)<(\nu+\delta)\theta'(0)$
and
$$
q'(r)=G(r,q(r)),
$$
where $G(r,q)$ is defined as
$$
G(r,q):=-\lambda\theta(r)+\frac{\nu^2}{\theta(r)}-\frac{q^2}{\theta(r)}.
$$
Note also that, since $R\in(0,\bar R)$, we have from $iii)$ of Theorem \ref{technical_sphere_0} that $\lambda=\bar\lambda+\epsilon$ for some $\epsilon>0$.

Let us compute
\begin{multline*}
G(r,(\nu+\delta)\theta'(r))-((\nu+\delta)\theta'(r))'=-\lambda\theta(r)+\frac{\nu^2}{\theta(r)}-(\nu+\delta)^2\frac{\theta'(r)^2}{\theta(r)}-(\nu+\delta)\theta''(r)\\
=-\epsilon\theta(r)-\frac{1}{\theta}\left(\bar\lambda\theta^2(r)-\nu^2+\nu^2\theta'(r)^2+\nu\theta(r)\theta''(r)\right)\\
-\frac{\delta}{\theta}\left(\theta'(r)^2(\delta+2\nu)+\theta''(r)\theta(r)\right)
\end{multline*}
The second summand is less or equal than zero by hypothesis, therefore we can choose $\delta\in(0,1-\nu)$ arbitrarily small so that 
$$
G(r,(\nu+\delta)\theta'(r))-((\nu+\delta)\theta'(r))'<0
$$
on $(0,R)$.
From Lemma \ref{tec} here below  with $v=q$ and $w=(\nu+\delta)\theta'$, $a=0$, $b=R$, we deduce that
$$
q<(\nu+\delta)\theta'
$$
as claimed.

\medskip

In the case $R=\bar R$, if $u=\theta^{\nu}$, then $q=\nu\theta'\leq\theta'$ on $(0,\bar R)$.

\end{proof}

\begin{lemma}\label{tec}
Suppose $v,w$ continuous on an interval $[a,b]$ and differentiable on $(a,b]$, and let $G:(a,b]\times\mathbb R\rightarrow\mathbb R$ continuous. Suppose that
$$
v(a)<w(a)
$$
and
\begin{equation}\label{tec1}
v'(r)-G(r,v(r))<w'(r)-G(r,w(r))
\end{equation}
on $(a,b]$. Then $v<w$ on $[a,b]$.
\end{lemma}
\begin{proof}
Assume that there exists $c\in(a,b]$ such that $v(c)=w(c)$ and $v-w<0$ on $[a,c)$. Then clearly $v'(c)\geq w'(c)$. On the other hand, inserting $v(c)=w(c)$ in \eqref{tec1}, we see that $v'(c)<w'(c)$. A contradiction.
\end{proof}

\subsection{Eigenfunctions of the Steklov problem}\label{rev_S}

The same results of the previous subsection hold also for the Steklov problem \eqref{AB}. For the reader's convenience, we briefly resume them, even if in this note we almost only consider the Steklov problem for domains in $\mathbb R^2$.

When looking for solutions on $B(x_0,R)$ of the form $f(r,t)=u(r)e^{ikt}$  we get the following singular Sturm-Liouville problem, for $k\in\mathbb Z$:
\begin{equation}\label{revolution_steklov}
\begin{cases}
u''+\frac {\theta'}{\theta} u'-\frac{(k-\nu)^2}{\theta^2}u=0 & {\rm in\ } (0,R),\\
u'(R)=\sigma u(R).
\end{cases}
\end{equation}
with the requirement that a solution is bounded near $0$ (and therefore, that it is vanishes at $0$ when $\nu\notin\mathbb Z$).

We have the following

\begin{lemme}\label{lemma_S} Let $\nu\in (0,\frac 12]$. Then problem \eqref{revolution_steklov} admits a unique solution $u_k(r)$ with $u_k(0)=0$ for all $k\in\mathbb Z$. Let us set
$$
\eta_k=\frac{u_k'(R)}{u_k(R)}.
$$
 If $\abs{k}\leq \abs{h}$ then:
$$
\eta_{k}\leq \eta_{h}.
$$
Then the set $\{\sigma_k(B(x_0,R),A_{x_0,\nu})\}_{k=1}^{\infty}$ of the Steklov eigenvalues on $B(x_0,R)$ coincides with $\{\eta_k\}_{k\in\mathbb Z}$. In particular, $\min_{k}\{\eta_k\}=\eta_0$, therefore
$$
\sigma_1(B(x_0,R),A_{x_0,\nu})=\eta_{0}.
$$
 The eigenfunctions corresponding to an eigenvalues $\sigma=\eta_k$ of \eqref{AB} are given by
$$
\psi_k(r,t)=u_k(r)e^{ikt}.
$$
In particular, an eigenfunction associated to $\sigma_1(B(x_0,R),A_{x_0,\nu})$  is given by $\psi_0(r,t)$ which is real and radial. The restrictions of $\{\psi_k(r,t)\}_k$ to $\partial B(x_0,R)$ form a orthonormal system in $L^2(\partial B(x_0,R))$
\end{lemme}
\begin{proof}
The proof is analogous of that of Lemma \ref{lemma_N}. It is sufficient to note that a bounded solution of \eqref{revolution_steklov} is of the form.
$$
u_k(r)=e^{\int_{R}^r\frac{|k-\nu|}{\theta(s)}ds}.
$$
The last statement is straightforward.
\end{proof}

\subsection{The magnetic Neumann spectrum of the unit disk}\label{disk}

We consider the particular case $M=B(x_0,1)$, where $B(x_0,1)$ is the unit disk in $\mathbb R^2$ centered at $x_0$. We recall that $A=\nu dt$ and we assume $\nu\in(0,\frac 12]$. Here $(r,t)$ are the standard polar coordinates in $\mathbb R^2$ centered at $x_0$. In this case $\theta(r)=r$.

\smallskip

We can describe the spectrum of \eqref{AB_N} on $B(x_0,1)$ more explicitly:

 \begin{thm}
Let $\nu\in(0,\frac 12]$. Then:
\begin{enumerate}[i)]
\item The spectrum of \eqref{AB_N} on $B(x_0,1)$ consists of the numbers $\lambda_{kj}$, $k\in\mathbb Z$, $j=1,2,...$, where $\sqrt{\lambda_{kj}}=z'_{|k-\nu|,j}$ and $z'_{\mu,j}$ denotes the $j$-th positive zero of $J_{\mu}'(\sqrt{\lambda})$. Here $J_{\mu}(z)$ denotes the Bessel function of the first kind and order $\mu$.

\item The eigenspace associated to $\lambda_{kj}$ is spanned by $\psi_{kj}(r,t)=J_{|k-\nu|}(\sqrt{\lambda_{kj}}r)e^{ikt}$.

\item The lowest eigenvalue is $\lambda_1(B(x_0,1),A_{x_0,\nu})=\lambda_{01}$ and a first eigenfunction is $\psi_{01}(r,t)=J_{\nu}(\sqrt{\lambda_{01}}\,r)$. It is real and radial.
\item The set of eigenfunctions $\{\psi_{kj}\}_{k,j}$ is complete in $L^2(B(x_0,1))$.
\end{enumerate}
\end{thm}

\begin{proof}
We consider \eqref{SL_N} with $\theta(r)=r$ and $R=1$. As in Subsection \ref{sub:rev_N}, when looking for solutions of the form $f(r,t)=u(r)e^{ikt}$ in polar coordinates $(r,t)$ , we obtain the following differential equation for $u$:
\begin{equation}\label{Bessel}
u''+\frac 1r u'+\left(\lambda-\frac{(k-\nu)^2}{r^2}\right)u=0.
\end{equation}
With the substitution $\sqrt{\lambda}r=z$, this equation can be recast to a standard Bessel equation. Therefore, a couple of linearly independent solutions of \eqref{Bessel} is given by $J_{|k-\nu|}(\sqrt{\lambda}r)$, $J_{-|k-\nu|}(\sqrt{\lambda}r)$, since by hypothesis $\nu\notin\mathbb Z$. Therefore 
$$
u(r)=b_{k1}J_{|k-\nu|}(\sqrt\lambda r)+b_{k2} J_{-|k-\nu|}(\sqrt\lambda r).
$$
Note that, while $J_{|k-\nu|}$ vanishes at the origin whenever $\nu\notin\mathbb Z$, the function $J_{-|k-\nu|}$ is unbounded near $r=0$. Any bounded solution of \eqref{Bessel} is given by setting $b_{k2}=0$. Therefore, any solution of \eqref{SL_N} with $\theta(r)=r$ has the form 
$$
u_k(r)=b_{k1}J_{|k-\nu|}(\sqrt\lambda r).
$$
We impose now the Neumann boundary condition $u_k'(1)=0$. We recall that for any $\mu>0$, the function $J_{\mu}$ has an infinite set of positive zeroes, denoted  by $z_{\mu j}$, $j=1,2,...$. Moreover, the function $J'_{\mu}$ an infinite number of positive zeroes as well, namely $z'_{\mu j}$, $j=1,2,...$. 

\smallskip

Then, points $i)$-$iv)$ follow from Lemma \ref{lemma_N} and Theorem \ref{complete} with $\sqrt{\lambda_{kj}}=z'_{|k-\nu|,j}$ and  $\psi_{kj}(r,t)=J_{|k-\nu|}(\sqrt{\lambda_{kj}}r)e^{ikt}$.

\end{proof}

If $\nu\in (0,\frac 12)$ then the set $\{|k-\nu|\}_{k\in\mathbb Z}$ is given by $\nu,1-\nu,1+\nu,2-\nu,2+\nu,...$. For example, when $\nu=\frac 14$ we have $\frac 14, \frac 34, \frac 54, \frac 74,...$
If $\nu\in(0,\frac 12)$, we immediately see that $\lambda_1(B(x_0,1),A_{x_0,\nu})$ is simple, and a first eigenfunction is real and radial:
$$
\psi_{01}(r,t)=J_{\nu}(\sqrt{\lambda_{01}} r).
$$
Such function is is positive and increasing on $(0,1)$.

\smallskip

When $\nu=\frac 12$, we have that the set of $\{|k-\nu|\}_{k\in\mathbb Z}$ is given by $\frac 12, \frac 12, \frac 32,\frac 32, \dots$,
and the first eigenvalue has multiplicity $2$ with eigenspace spanned by the eigenfunctions with $k=0$ and $k=1$:
$$
\twosystem
{\psi_{01}(r,t)=J_{\frac 12}(\sqrt{\lambda_{01}}r)}
{\psi_{11}(r,t)=J_{\frac 12}(\sqrt{\lambda_{01}}r)e^{it}}
$$
and $\lambda_{01}=\lambda_{11}=z'_{\frac 12,1}$ is the first zero of $J'_{\frac 12}$.

\subsection{The magnetic Steklov spectrum of the unit disk}\label{disk_S}

In this section we investigate the spectrum of problem \eqref{AB} on $B(x_0,1)\subset\mathbb R^2$. Also in this section we assume that $\nu\in(0,\frac 12]$.

We apply the results of Subsection \ref{rev_S}  with $\theta(r)=r$ and $R=1$.

\begin{thm}
Let $\nu\in(0,\frac 12]$. Then:
\begin{enumerate}[i)]
\item The spectrum of \eqref{AB} on $B(x_0,1)$ consists of the numbers $\eta_k=|k-\nu|, \quad k\in\mathbb Z$.
\item The eigenspace associated to $\eta_k$ is spanned by $\psi_{k}(r,t)= r^{|k-\nu|}e^{ikt}$.
\item The lowest eigenvalue is $\sigma_1(B(x_0,1),A_{x_0,\nu})=\eta_{0}=\nu$. A first eigenfunction is $\psi_1(r,t)=r^{\nu}$.
It is real and radial.
\item The restrictions of $\{\psi_{kj}\}_{k,j}$ to $\partial B(x_0,1)$ is complete in $L^2(\partial B(x_0,1))$.
\end{enumerate}
\end{thm}
\begin{proof}
It is sufficient to note that any solution to the differential equation in \eqref{revolution_steklov} is of the form
$$
u_k(r)=b_{k1}r^{|k-\nu|}+b_{k2}r^{-|k-\nu|},
$$
and in order to have a solution of \eqref{revolution_steklov} we need to impose $b_{k2}=0$. The theorem now easily follows from Lemma \ref{lemma_S}.
\end{proof}

When $\nu\in(0,\frac{1}{2})$, the first eigenvalue is $\nu>0$. It is positive and simple and a corresponding eigenfunction is given by $r^{\nu}$ and it is radial, positive and increasing on $(0,1)$.

\smallskip

We can list the Steklov eigenvalues $\{\sigma_k(B(x_0,1),A_{x_0,\nu})\}_{k=1}^{\infty}$ as follows:
$$
\nu,1-\nu,1+\nu,2-\nu,2+\nu,...
$$ 

If $\nu=\frac{1}{2}$ all eigenvalues are double and are given by $\frac{1}{2},\frac{1}{2},\frac{3}{2},\frac{3}{2},...$.

In particular, $\sigma_1(B(x_0,1),A_{x_0,\nu})=\sigma_2(B(x_0,1),A_{x_0,\nu})=\frac{1}{2}$ and two linearly independent eigenfunctions are $u_1(x)=\sqrt{|x|}, u_2(x)=\sqrt{|x|}\left(\frac{x_1}{|x|}+i\frac{x_2}{|x|}\right)$. However, also in this case there exist a radial, real eigenfunction associated with $\sigma_1(B(x_0,1),A_{x_0,\nu})$, positive and increasing on $(0,1)$.

\section{Conformal invariance of magnetic energy}\label{sec:conformal}

Let  $(\Omega_1,g)$, $(\Omega_2,g)$ bet two $2$-dimensional manifolds, and let
$$
\Phi: (\Omega_1,g_1)\to (\Omega_2,g_2)
$$
be a conformal map between them. We can assume then that $\Phi^{\star}g_2=e^{2f}g_1$, where $f$ is the conformal factor. Let $\omega$ be a $1$-form on $\Omega_2$. We start from:

\begin{lemma}\label{conformaldelta}
We have
\begin{enumerate}[i)]
\item In the above notation
$$
\delta_{g_1}\Phi^{\star}\omega=e^{2f}\Phi^{\star}(\delta_{g_2}\omega).
$$
In particular, $\omega$ is co-closed if and only if $\Phi^{\star}\omega$ is co-closed.

\item For any complex $1$-form $\Omega$ on $\Omega_2$, one has:
$$
\int_{\Omega_1}\abs{\Phi^{\star}\omega}_{g_1}^2d\mu_{g_1}=\int_{\Omega_2}\abs{\omega}^2_{g_2}d\mu_{g_2}.
$$
\end{enumerate}
\end{lemma}

\begin{proof}
The first fact is standard; for the second, fix a $g_1$- orthonormal basis $(e_1,e_2)$ and let
$$
E_1=d\Phi(e^{-f}e_1), \quad E_2=d\Phi(e^{-f}e_2).
$$
Using these orthonormal frames to compute the norms we end-up with the identity:
$$
\abs{\Phi^{\star}\omega}_{g_1}^2=e^{2f}(\abs{\omega}^2_{g_2}\circ\Phi).
$$
Integrating the identity on $\Omega_1$ and using the change of variables formula, we obtain the assertion.
\end{proof}

Now fix a potential one-form $A$ on $\Omega_2$. We get a potential one-form $\hat A=\Phi^{\star}A$ on $\Omega_1$ by pull-back. Recall the magnetic gradient on $\Omega_2$:
$$
d^Au=du-iuA
$$
for a complex valued function $u$ on $\Omega_2$. Consider the function
$$
\hat u=u\circ\Phi=\Phi^{\star}u
$$
on $\Omega_1$. Then, since $d\Phi^{\star}=\Phi^{\star}d$:
$$
d^{\hat A}\hat u=d^{\Phi^{\star}A}\Phi^{\star}u
=d\Phi^{\star}u-i\Phi^{\star}u\Phi^{\star}A
=\Phi^{\star}\Big(du-iuA\Big)
=\Phi^{\star}d^Au
$$
Applying the $L^2$ invariance property to $\omega=d^Au$ we conclude with the following fact, expressing the conformal invariance of the magnetic energy.

\begin{prop}\label{cienergy} Let $\Phi:\Omega_1\to\Omega_2$ be a conformal map between 2-manifolds, let $A$ be a potential one-form on $\Omega_2$ and let $\hat A=\Phi^{\star}A$.
For any complex function $u$ on $\Omega_2$, let $\hat u=u\circ\Phi$. Then we have:
$$
\int_{\Omega_1}\abs{d^{\hat A}\hat u}^2d\mu_1=\int_{\Omega_2}\abs{d^Au}^2 d\mu_2.
$$
\end{prop}

\subsection{Aharonov-Bohm potentials}

Let now $(\Omega,x_0)$ be a simply connected plane domain punctured at $x_0$, with Aharonov-Bohm potential having pole at $x_0$ and flux $\nu$. We can assume without loss of generality that $x_0$ is the origin. Then:
\begin{equation}\label{ABpotential}
A=\dfrac{\nu}{r^2}(-ydx+xdy).
\end{equation}
Take the unit disk $D$ centered at the origin. By the standard theory, there is a conformal map
$
\Phi:\Omega\to D
$
fixing the origin. Note that the form $\Phi^{\star}A$ on $\Omega$ is closed (clear) and co-closed (from previous section, because $\Phi$ is conformal). Then,
$\Phi^{\star}A$ is harmonic.
It is also clear that $\Phi^{\star}A$ has flux $\nu$ around the pole, the origin of $\real 2$ (by elementary change of variable in dimension one). It follows that $\Phi^{\star}A$ differs from  $A$ by an exact $1$-form, and gauge invariance applies. Precisely:
\begin{lemma}\label{gaugeinvariance} Let $\Omega$ be a plane domain and let $D$ be the unit disk, both punctured at $O\in \Omega$. Let $\Phi:\Omega\to\ D$ be the unique conformal map 
fixing $O$. If $A$ is the canonical Aharonov-Bohm potential  with pole $O$ and flux $\nu$, as in \eqref{ABpotential}, then
$$
\sigma_k(\Omega,A)=\sigma_k(\Omega,\Phi^{\star}A)
$$
for all $k$.
\end{lemma}

\section*{Acknowledgements}\label{ackref}
The first author acknowledges support of the SNSF project ‘Geometric Spectral Theory’, grant number 200021-19689. The second and the third author are members of the Gruppo Nazionale per le Strutture Algebriche, Geometriche e le loro Applicazioni (GNSAGA) of the I\-sti\-tuto Naziona\-le di Alta Matematica (INdAM).


\addcontentsline{toc}{chapter}{Bibliography}
\bibliographystyle{plain}
\bibliography{biblioCPSmagnetic1}
\end{document}